\documentclass[a4paper]{amsart}
\usepackage{graphicx} %

\usepackage[shortlabels]{enumitem}
\usepackage[utf8]{inputenc}
\usepackage{amsmath}
\usepackage{amssymb}
\usepackage{amsthm}
\usepackage{tikz}
\usepackage{graphicx}
\usepackage{float}
\usepackage{svg}
\usepackage{verbatim}
\usepackage{url}
\usepackage{mathtools}
\usepackage{microtype}
\usepackage{thm-restate}
\usepackage{enumitem}
\usepackage[normalem]{ulem}
\usepackage[hidelinks]{hyperref}
\usepackage[nameinlink]{cleveref}
\usepackage{caption}
\title[Monodromies of surfaces, right-veeringness, and 
primeness of links]
{Monodromies of surfaces in 3-manifolds, right-veeringness, and %
primeness of links}
\author{Peter Feller}
\author{Lukas Lewark}
\author{Miguel {Orbegozo Rodriguez}}
\address{Université de Neuchâtel, %
Rue Emile-Argand 11, 2000 Neuchâtel, Switzerland}
\email{peter.feller@unine.ch}
\urladdr{\url{https://www.unine.ch/math/en/pfeller/}}
\address{Department of Mathematics, ETH Zurich, 8092 Zurich, Switzerland}
\email{lukas.lewark@math.ethz.ch}
\urladdr{\url{https://people.math.ethz.ch/~llewark/}}
\address{Université de Neuchâtel, %
Rue Emile-Argand 11, 2000 Neuchâtel, Switzerland}
\email{miguel.orbegozorodriguez@gmail.com}
\urladdr{\url{https://sites.google.com/view/miguel-orbegozo-rodriguez/home}}
\let\cref\Cref
\Crefname{enumi}{}{}

\theoremstyle{plain}
\newtheorem{theorem}{Theorem}[section]
\theoremstyle{definition}
\newtheorem{definition}[theorem]{Definition}
\theoremstyle{plain}
\newtheorem{lemma}[theorem]{Lemma}
\newtheorem{claim}[theorem]{Claim}

\theoremstyle{definition}
\newtheorem{example}[theorem]{Example}

\newenvironment{Example}{\begin{example}\rm}{\end{example}}
\theoremstyle{plain}
\newtheorem{proposition}[theorem]{Proposition}
\newtheorem{corollary}[theorem]{Corollary}
\newtheorem*{composition}{Composition formula}
\newtheorem*{characterization}{Characterization}
\newtheorem*{vprime}{Visual primeness}
\theoremstyle{remark}
\newtheorem{remark}[theorem]{Remark}

\newcommand{\qua}{\hskip 0.4em \ignorespaces}
\def\arxiv#1{\relax\ifhmode\unskip\qua\fi
\href{http://arxiv.org/abs/#1}%
{\tt arXiv:\penalty -100\unskip#1}}

\def\MR#1{\relax\ifhmode\unskip\qua\fi
\href{https://mathscinet.ams.org/mathscinet-getitem?mr=#1}{\tt MR#1}}
\def\ZB#1{\relax\ifhmode\unskip\qua\fi
\href{https://zbmath.org/?q=an:#1}{\tt Zbl\:#1}}
\def\xox#1{\csname xx#1\endcsname}

\renewenvironment{thebibliography}[1]{
  \begin{oldthebibliography}{#1}\small
    \setlength{\itemsep}{.5ex}
    \setlength{\parskip}{0em}
}
{
  \end{oldthebibliography}
}

\newcommand{\mor}[1]{{\color{red}#1}} %

\newcommand{\A}{\mathcal{A}}
\newcommand{\R}{\mathbb{R}}
\newcommand{\Z}{\mathbb{Z}}

\begin{document}

\begin{abstract}
Extending the notion of monodromies associated with open books of $3$-manifolds, we consider monodromies for all incompressible surfaces in $3$-manifolds as partial self-maps of the arc set of the surfaces.
We use them to develop a primeness criterion for incompressible surfaces constructed as iterative Murasugi sums in irreducible $3$\nobreakdash-manifolds.

We also consider a suitable notion of right-veeringness for monodromies of incompressible surfaces.  We show strongly quasipositive surfaces are right-veering, thereby generalizing the corresponding result for open books and providing a proof that does not draw on contact geometry.
In fact, we characterize when all elements of a family of incompressible surfaces that is closed under positive stabilization are right-veering. The latter also offers a new perspective on the characterization of tight contact structures via right-veeringness as first established by Honda, Kazez, and Mati\'c. 

As an application to links in~$S^3$, we prove visual primeness of a large class of links, the so-called alternative links.
This subsumes all prior visual primeness results related to Cromwell's conjecture. The application is enabled by the fact that all links in $S^3$ arise as the boundary of incompressible surfaces, whereas classical open book theory is restricted to fibered links---those links that arise as the boundary of the page of an open book.
\end{abstract}

\maketitle

\section{Introduction}

We study links $L$ in a $3$-manifold $M$ via incompressible surfaces with boundary~$L$. For simplicity of exposition, we focus on $M=S^3$ in the introduction. If $L$ is the binding of an open book (in other words, $L$ is a fibered link),
the theory of open books provides an extensive toolbox for studying properties of $L$; e.g., behavior under Murasugi sums~\cite{Stallings, Gabai, Gabaiother}, hyperbolicity \cite{Thurston}, and homotopy ribbonness \cite{Casson-Gordon}.
Questions about $L$ can be phrased in terms of the monodromy~$\varphi$ of the open book; concretely, its effect on isotopy classes of properly embedded arcs.
In this paper we associate to every incompressible surface $\Sigma$ a monodromy (which takes the form of a partially defined self-map on the arc graph of $\Sigma$), and we use this monodromy to study~$L$, similarly as in the fibered case. In particular, we investigate the behavior of these monodromies under the operation of Murasugi sums, and we study so-called strongly quasipositive surfaces (a notion motivated by contact topology and complex algebraic plane curves) and the properties of their associated monodromies. A showcase application will be to the question of visual primeness of links around Cromwell's conjecture~\cite{Cromwell_93}.

\subsubsection*{Murasugi summing}
All fiber surfaces in $S^3$ arise from the trivial open book (the unknot with a disk as Seifert surface) by iterative plumbing and deplumbing of Hopf bands. These operations, also known as stabilization and destabilization, respectively, are generalized by what is called a Murasugi sum of surfaces.
Given two surfaces $\Sigma_1\subset M_1$ and $\Sigma_2\subset M_2$ (and a summing region, i.e.~a polygon embedded in each surface), their Murasugi sum is a new surface $\Sigma\subset M_1\#M_2$; see~\Cref{sec:Murasugisum} for details.
The Murasugi sum $\Sigma$ is a fiber surface if and only if $\Sigma_1$ and~$\Sigma_2$ are fiber surface, and, when this is the case, the monodromy $\varphi$ is the composition of the monodromies of the original fiber surfaces ~\cite{Stallings, Gabai, Gabaiother}: $\varphi=\Phi_1\circ \Phi_2$, where $\Phi_i$ denotes the natural extension by the identity of the monodromy of $\Sigma_i$ to a diffeomorphism of~$\Sigma$. This description of the monodromy in terms of the monodromy of the Murasugi summands is a fundamental tool in describing and understanding fibered links.

Our key result concerning Murasugi sums is the version of this composition formula for the monodromy of any Murasugi sum of incompressible surfaces. In particular, we define the \emph{(partial) monodromy $\phi$} of an incompressible surface in a 3--manifold $M$. Moreover, much like in the fibered case, given a Murasugi sum $\Sigma \subset M_1 \# M_2$ of $\Sigma_1 \subset M_1$ and $\Sigma_2 \subset M_2$, an ``extension by the identity to $\Sigma$" $\Phi_i$ of $\phi_i$ can also be defined.

\begin{composition}[\Cref{lemma:diskcomposition} and \Cref{rem:diskcomp}]
    Let $\Sigma_1\subset M_1$ and $\Sigma_2 \subset M_2$ be incompressible surfaces, and $\Sigma \subset M_1 \# M_2$ a Murasugi sum. Then, the monodromy of $\Sigma$ is given by $\phi = \Phi_1 \circ \Phi_2.$
\end{composition}

This allows us to prove a primeness criterion in \Cref{crit}, similar to the one in \cite{FLOR} for open books. In turn, we use this criterion to establish, in \Cref{thm:treeofveeringsurfaces}, primeness of a large class of surfaces, given by tree-guided Murasugi sums of right- and left-veering surfaces; see \Cref{sec:trees} for definitions and a precise statement of this primeness criterion. While we do not provide the exact statement here, at the end of the introduction we discuss the application that was our original motivation: visual primeness of a large class of link diagrams as predicted by Cromwell's conjecture.

\subsubsection*{Right-veeringness}

When considering monodromies of fiber surfaces, one of the basic properties to study is the notion of \emph{right-veeringness} %
(see \Cref{sec:Prelims} for a definition). A result of Hedden~\cite{Hedden} for fibered strongly quasipositive links in~$S^3$, together with the characterization of tight contact structures in terms of right-veering open books by Honda, Kazez, and Mati\'{c}~\cite{HKM}, shows that fiber surfaces of strongly quasipositive links in $S^3$ are right-veering. 

In our setup, for incompressible surfaces that are not necessarily fiber surfaces, we discuss an appropriate notion of right-veeringness (originally considered in~\cite{HKMPartial}). %
Using the fact that strong quasipositivity is preserved by positive Hopf plumbing, we are able to prove the analogous result (see \Cref{prop:sqp=>rv}):
strongly quasipositive surfaces in $S^3$ are right-veering.

Moreover, we note that this reproves the result in the fibered case without the use of contact geometry.
In fact, \Cref{prop:sqp=>rv} is a consequence of the main theorem concerning right-veeringness in this article, which we consider to be of independent interest. %
It is a characterization of which families of incompressible surfaces (considered up to isotopy) that are closed under positive Hopf plumbing are right-veering.

\begin{characterization} [%
\cref{thm:rvnesscrit} and~\cref{rmk:soberingarc}]
In a $3$-manifold,
a family of compact incompressible surfaces that is closed under positive Hopf plumbing consists only of right-veering surfaces if and only if no surface in that family has an essential left-veering arc with interior disjoint from the arc's image under the monodromy.
\end{characterization}
As a consequence of this characterization, we provide a different proof of the fact that a contact structure on a $3$-manifold is tight if and only if all its open books are right-veering with minimal reliance on tools from contact geometry.
The latter is in spirit and technique rather similar to an alternative proof provided by~\cite{Wand}; see~\cref{rem:wand}.

\subsubsection*{Visual primeness}
Our main application of our primeness criterion for Murasugi sums (\cref{crit}, \cref{thm:treeofveeringsurfaces})-%
and right-veeringness of strongly quasipositive surfaces in $S^3$ (\cref{prop:sqp=>rv}) is to the question of visual primeness. Cromwell conjectured~\cite{Cromwell_93} that primeness of a link can be read off from a diagram, if applying Seifert's algorithm to the diagram yields a minimal genus Seifert surface. He called this property \emph{visual primeness} (a precise definition will be given in~\cref{sec:alternative}).
We prove a visual primeness result for a large class of links that strictly subsumes the visual primeness results previously established in the literature.

\begin{vprime}[\cref{thm:alternativevp}]
If a non-prime link has an alternative diagram,
then the decomposition of the link as a connected sum is visible in that diagram.
\end{vprime}

We defer the definition of alternative links and diagrams, introduced in~\cite{zbMATH03854015}, to~\cref{sec:alternative}. Rather than focusing on what family of links\footnote{A non-representative poll among low-dimensional topologists indicates that alternative links are a little known class of links, and their definition is what many assume is the definition of homogeneous links.} is treated in \cref{thm:alternativevp}, we think of it as proof of concept that the tools developed in this text are a contender to attack Cromwell's conjecture in full generality.
In particular, \cref{thm:alternativevp} gives a unifying proof for a class of links that strictly contains all previous visual primeness results (\cite{Menasco_84} for alternating links, \cite{Cromwell_93} for positive braids (compare also \cite{zbMATH07504315}), and \cite{zbMATH01807859} for positive links; see also \Cref{cor:poslinksarevisprime}). Finally, the above also recovers the main result of~\cite{FLOR}, as link diagrams stemming from homogeneous braid diagrams are special cases of alternative link diagrams. A crucial difference to that work is that alternative links
form a much vaster class of links that are, in general, not fibered. Hence, the technology developed in this paper for all links (and their incompressible Seifert surfaces), is needed, while in~\cite{FLOR}, we relied on standard open book technology. 
\subsubsection*{Structure of the paper}

In \Cref{sec:Prelims}, we give the definition of monodromy of an incompressible surface as well as definitions of related notions such as fixed and right-veering arcs in an incompressible surface. We also define Murasugi sums.
In \Cref{sec:Murasugisum}, we prove the composition formula (the monodromy of a Murasugi sum of surfaces is the composition of the monodromies of the summands), as well as the primeness criterion \Cref{crit}.
Then, in \Cref{sec:trees}, we use this criterion to establish primeness of tree guided Murasugi sums of right- and left-veering surfaces (\cref{thm:treeofveeringsurfaces}), and of arborescent links in~$S^3$.
In \Cref{sec:sqp}, we prove that strongly quasipositive surfaces in $S^3$ are right-veering; in fact, we characterize right-veeringness of families of surfaces that are closed under positive stabilization, and we use the latter to give a different proof of Honda, Kazez, and Mati\'c's celebrated characterization of tightness of a contact structure.
Finally, in \Cref{sec:alternative}, we use the tools from \Cref{sec:trees,sec:sqp} to prove that alternative links are visually prime.

\subsubsection*{Acknowledgments} MOR acknowledges support by the DFG grant~513007277.

\section{Preliminaries on incompressible surfaces and their %
monodromies} \label{sec:Prelims}
All manifolds are smooth, oriented and are allowed to have boundary, unless specified otherwise. For a manifold~$M$, we denote by $\overline{M}$ the same manifold with the opposite orientation. Some of our constructions will create manifolds with corners, in which case we will implicitly smooth the corners whenever necessary.

The goal of this section is to build up the language of partial monodromies %
of incompressible surfaces in 3-manifolds as we need it. This is all aimed to mirror the case of how monodromies (of open books) relate to fiber surfaces in 3-manifolds.
Rather than thinking of partial monodromies as diffeomorphisms with domain and target some subsurface (an approach developed elsewhere under the name ``partial open books''~\cite{HKMPartial,EtguOzbaugci}; see \cref{rmk:phitopartialopenbook} for the connection), we focus on isotopy classes of arcs and how they are mapped. 

\subsection{Monodromies as partial self-maps of the arc set}

We let the \emph{arc set} of~$\Sigma$, denoted by~$\mathcal{A}(\Sigma)$, be the set of isotopy classes of \emph{arcs}---proper smooth embeddings~$[0,1]\to\Sigma$, which we readily conflate with their images, remembering orientation---in a surface with boundary~$\Sigma$, where isotopies fix boundary pointwise. (Note that, crucially, isotopies are fixed on the boundary, hence our arc sets are typically uncountable.\footnote{A so inclined reader might want to avoid this, e.g.~by choosing to pick a marked point on each boundary component and ask for arcs to start and end at such marked points. This approach leads to an equally strong setup. %
})

If $\Sigma'\subset \Sigma$ is a subsurface, we consider the restriction, also known as the subsurface projection, of (isotopy classes of) arcs in $\Sigma$ to $\Sigma'$: for $a\in \mathcal{A}(\Sigma)$,
\[a\cap \Sigma'\coloneqq \left\{[\delta]\in \mathcal{A}(\Sigma')\ \middle\vert
\begin{array}{l}\text{$\gamma$ is a representative of $a$ that intersects $\partial \Sigma'$ minimally}\\
\text{and $\delta$ is a connected component of $\gamma\cap \Sigma'$}
\end{array}\right\}.\]

For an incompressible surface $\Sigma$ in a $3$-manifold~$M$, we denote by $\phi(a)\in \mathcal{A}(\Sigma)$ the isotopy class of arcs obtained by ``bringing $a$ to the other side (by a product disk)'', if it exists.
We use the convention that $a$ is on the positive side of~$\Sigma$, while $\phi(a)$ is on the negative side.
By convention, writing $\phi(a)$ in particular means that $a$ goes to the other side.

We make this precise using the language of product disks as introduced by Gabai~\cite{Gabai_FoliationsII}. However, we give a definition of a product disk without relying on the sutured manifold terminology.
Of course, it is natural to phrase things in the language of sutured manifolds. Indeed, in that language~\cite{Gabai_FoliationsI,Gabai_FoliationsII}, given a balanced sutured manifold $(M,\gamma)$---such as the exterior of a compact incompressible surface without closed components in a $3$-manifold with its canonical sutures---$\phi$ can be understood as a partial map $\phi\colon \mathcal{A}(R^+(\gamma))\to \mathcal{A}(R^-(\gamma))$, mapping an arc in $R^+(\gamma)$ to the arc in $R^-(\gamma)$ if their union is the boundary of a product disk.
Still, we avoid this language, as our point is to not actually use product disks as decomposition surfaces, but instead to collect the information which arcs feature as the boundary of product disks in the arc set. Moreover, our focus is on surfaces and the operation of Murasugi sums on surfaces. 

A \emph{product disk} is a smooth map $\Delta\colon[0,1]\times[0,1]\to M$ with the following properties. The restriction to $[0,1]\times\{0\}$ and $[0,1]\times\{1\}$ defines proper arcs in~$\Sigma$. The restrictions to $\{0\}\times [0,1]$ and $\{1\}\times [0,1]$ are constant, so in particular the two arcs have the same start and end points. The restriction to $(0,1)\times(0,1)$ is a diffeomorphism onto its image and the derivative at $(t,i)$ for $t\in(0,1)$ and $i\in\{0,1\}$ with respect to the second coordinate yields a tangent vector $v\in T_{\Delta(t,i)}M$ that points positively out of~$\Sigma$. Here, $v\in T_pM$ for $p\in\Sigma$ is said to \emph{point positively out of $\Sigma$}, if for a basis $(b_1,b_2)$ of $T_p\Sigma$ that defines the orientation of $T_p\Sigma$, we have that $(b_1,b_2,v)$ is a basis of $T_pM$ that defines the orientation of~$T_pM$.

\begin{definition}[The (partial) monodromy of $\Sigma$]
Let $\Sigma$ be an incompressible surface in a $3$-manifold~$M$.
We define a partial map\footnote{A partial map $f$ from a set $X$ to a set $Y$ is a function from a subset $\mathrm{domain}(f)\subseteq X$ to~$Y$.} $\phi_\Sigma\colon \A(\Sigma)\to \A(\Sigma)$ called the \emph{partial monodromy} of $\Sigma$ as follows.
The domain of $\phi_\Sigma$ is formed by those $a\in \A(\Sigma)$
for which there exists a product disk $\Delta$ such that \[a=\left[[0,1]\to \Sigma,t\mapsto \Delta(t,0)\right].\]
For such $a$, we define \[\phi_\Sigma(a)\coloneqq \left[[0,1]\to \Sigma, t\mapsto \Delta(t,1)\right].\]
When the surface $\Sigma$ is clear from the context, we simply write $\phi$ instead of $\phi_\Sigma$ and
when the subtlety of whether $\phi$ is defined on all of $\mathcal{A}(\Sigma)$ is not relevant, we refer to $\phi$ as the \emph{monodromy} of~$\Sigma$.
\end{definition}
\begin{remark}
By incompressibility of~$\Sigma$, $\phi_\Sigma$ is well-defined (as a partial map) and in fact yields a bijection from $\mathrm{domain}(\phi_\Sigma)$
to~$\mathrm{image}(\phi_\Sigma)$.
We write $\phi_\Sigma^{-1}\colon \A(\Sigma)\to\A(\Sigma)$ for the partial map given by $\mathrm{domain}(\phi_\Sigma^{-1})\coloneqq \mathrm{image}(\phi_\Sigma)\ni a'=\phi_\Sigma(a)\mapsto a$,
and note that
$\phi_{\Sigma}^{-1}=\phi_{\overline{\Sigma}}$, if we canonically identify $\A(\Sigma)$ with~$\A\left(\overline{\Sigma}\right)$.
\end{remark}

Up to reparametrization of the domain, a product disk is determined by its image in~$M$. We will denote (the images of) product disks by~$D$, which formally are bigons for which both corners have full angle, and denote the images of the arcs obtained by restriction to $[0,1]\times \{0\}$ and $[0,1]\times \{1\}$ by $\partial_0 D$ and~$\partial_1 D$, respectively. As an illustration and to fix conventions, we provide the most basic non-trivial example.

\begin{example}[The positive Hopf band]\label{Ex:HB}\rm
We consider the positive Hopf band
$H$ in~$\R^3\subset S^3$, as depicted in \Cref{fig:Hopfband},
\begin{figure}[ht]
    \centering
    \includegraphics[width=0.4\linewidth]{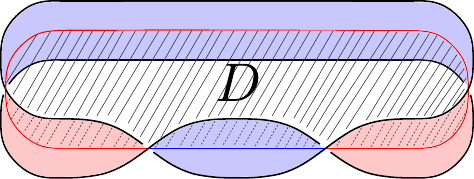}
    \caption{A positive Hopf band $H$ and a product disk $D$, with $\partial_0 D$ in blue and $\partial_1 D$ in red.}
\label{fig:Hopfband}
    \end{figure}
and indicate a product disk $D$
for $H$ for which $\partial_0 D$ and $\partial_1 D$ form cocores of~$H$. In particular, we see that the monodromy $\phi$ is defined for all arcs and is in fact induced (via an orientation preserving parametrization $H \cong S^1\times [0,1]$) by the positive Dehn twist $T\colon  S^1\times [0,1]\to S^1\times [0,1], (z,t)\mapsto (e^{2\pi it}z,t)$.
\end{example}

\begin{example}[Fiber surfaces]\label{ex:fibersurfaces}\rm 
Let $\Sigma$ be a compact surface with non-empty boundary in a closed $3$-manifold~$M$. If $\Sigma$ is a fiber surface (a.k.a.~the page of an open book), then $\phi$ is defined for all~$a \in \A(\Sigma)$. Indeed, $\Sigma$ being a fiber surface means that $M \setminus \nu (\partial \Sigma)$ is a mapping torus $\frac{\Sigma \times [0,1]}{(x,1) \sim (\varphi(x),0)}$ for some diffeomorphism $\varphi$ of $\Sigma$ fixing the boundary pointwise. Thus, for each $a \in \A(\Sigma)$, if $\gamma$ is a representative of $a$, then $\frac{\gamma \times [0,1]}{\sim_{\varphi}}$ is a product disk witnessing $\gamma$ going to $\varphi(\gamma)$, so $\phi(a)$ is the isotopy class of the arc $\varphi(\gamma)$.  
\end{example}

In fact, there is a sort of converse of \cref{ex:fibersurfaces}: if $\phi$ is defined on all of $\mathcal{A}(\Sigma)$, then it is induced by a self-diffeomorphism and, under an irreducibility assumption, $\Sigma$ is a fiber surface. And, if $\phi$ is only a partial map, then it is induced by a partial open book in the flavor of~\cite{HKMPartial}. We explain this in the following remark, without going into details, as this perspective is not the focus here.
\begin{remark} \label{rmk:phitopartialopenbook}
Let $M$ be a closed connected $3$-manifold and let $\Sigma\subset M$ be an incompressible compact connected surface with non-empty boundary. 
Assume that its monodromy $\phi$ is defined on all of~$\A(\Sigma)$; as is e.g.~the case if $\Sigma$ is a fiber surface.
Then, there exists a  diffeomorphism (unique up to isotopy) $\varphi\colon\Sigma\to\Sigma$
that induces $\phi$, i.e.\ $\phi([\alpha])=[\varphi(\alpha)]$ for all arcs $\alpha$ in $\Sigma$. %
Indeed, by taking a cutting system of arcs for $\Sigma$ (disjoint arcs $\alpha_1,\dots,\alpha_l$,
cutting along which turns $\Sigma$ into disks), and defining $\Phi$ on each of them such that $[\varphi(\alpha_i)]=\phi([\alpha_i])$
and extending over the remaining disks, one finds $\varphi$.

Furthermore, if the monodromy $\phi$ is defined on all of~$\A(\Sigma)$, $\Sigma$ is a fiber surface if and only if $M\setminus \Sigma$ is irreducible. (In which case, $\Sigma$ corresponds to the page of the open book with monodromy given by $\varphi$). Indeed, if $M\setminus \Sigma$ is irreducible, we may take a cutting system, consider a neighborhood $N$ of $\Sigma$ with disjoint product disks for all the arcs in the cutting system, note that $\partial N$ consists of spheres and extend to a fibration of $M$ since all components of $M\setminus N$ are balls by the irreducibility assumption. If $M\setminus \Sigma$ is not irreducible, one may take away the arising summands, and the argument applies for the remaining $3$-manifold $\widetilde{M}$ in which $\Sigma$ is now a fiber surface with monodromy $\varphi$, hence the same cannot be true for $M$ (as it is not diffeomorphic to $\widetilde{M}$).

Similarly, if $\phi$ is not defined on all of $\mathcal{A}(\Sigma)$, one finds unique (up to isotopy) maximal (with respect to inclusion and such that each component contains at least one boundary component of $\Sigma$) full compact subsurface $P$ and $\varphi(P)$ of $\Sigma$ and a diffeomorphism $\varphi\colon \Sigma\to \varphi(P)$ that induces $\phi$ (meaning, $\phi(a)$ is defined if $a=[\alpha]$ for an arc $\alpha$ in $\Sigma$ that is contained in $P$ and $\phi([\alpha])=[\varphi(\alpha)]$ for all such $\alpha$s). To find $P$, take a maximal set of pairwise non-parallel essential arcs $\alpha_1,\dots,\alpha_l$ such that $\phi$ is defined on each $[\alpha_i]$, and take $P$ to be a closed neighborhood of
\[\partial \Sigma\cup\bigcup_{i=1}^{l}{\alpha_i}\cup D_1\cup\dots D_k,\] where the $D_i$ are the disk components of $\Sigma\setminus \bigl(\partial P\cup\bigcup_{i=1}^{l}{\alpha_i}\bigr)$. Similarly, $\varphi(P)$ is found with a collection of arcs that are representatives of $\phi[\alpha_i]$ and $\varphi$ is defined as needed to induce $\phi$. For the uniqueness, we note that if we have a collection of arcs $\{a_i\} \subset \A(\Sigma)$ that cut out a disk from $\Sigma$, and $\phi(a_i)$ is defined for all $i$, then $\phi$ is defined for every $a\in\mathcal{A}(\Sigma)$ given by an arc contained in this disk, and indeed completely determined by the $\phi(a_i)$. We will use this in the proof of \Cref{thm:alternativevp}.

The triple $(\Sigma,P,\varphi)$ amounts essentially to a partial open book in the sense of~\cite{HKMPartial} (only ``essentially'' since in~\cite{HKMPartial}, compare also~\cite{EtguOzbaugci}, $P$ amounts to being a neighborhoods of a union of disjoint arcs, while to have uniqueness of $(\Sigma, P,\varphi)$ associated to $\Sigma$, we opt to extend $\varphi$ over disk components of the complement). As seen above, the association of a well-defined $P$ and $\varphi$ with the incompressible surface $\Sigma$ is a bit subtle. So while what we pursue in this article can be phrased in terms of this partially defined diffeomorphism $\varphi$, e.g.~the definition of right-veering given below does indeed correspond to the notion of right-veeringness for partial open books as discussed in~\cite{HKMPartial}, we abstain from further discussing this perspective, as for our purposes, we find the perspective as a partially defined map on the arc set to be the better fit.
\end{remark}

\subsection{Fixed arcs and right-veeringness}

For an arc $a$ on some surface $\Sigma$ that does go to the other side, i.e.~$a\in\mathrm{domain}(\phi)$, we can say whether it is fixed, or right- or left-veering as in the fibered case (which is a case where $\phi$ is defined on all of~$\mathcal{A}(\Sigma)$).
For this, recall that for $p\in\partial\Sigma$ the set $\{a\in\mathcal{A}(\Sigma)\mid a\text{ starts at }p\}$ features a total order given by one arc lying to the right of the other: $a<b$ if $a\neq b$ and transversely and minimally intersecting representatives $\alpha$ and $\beta$ of $a$ and $b$, respectively, are such that the tangent vector of $\beta$ at $p$ is to the right of the tangent vector of $\alpha$. We write $a\leq b$ if $a<b$ or $a=b$.

\begin{definition}
    Let $a\in \A(\Sigma)$ be such that $a$ goes to the other side, i.e. we have $a\in \mathrm{domain}(\phi)\subset\A(\Sigma)$.
    If~$\phi(a)=a$, we say that $a$ is \emph{fixed}.
    If not, we say that $a$ \emph{veers strictly to the right}, if $a<\phi(a)$, and we say that $\alpha$ \emph{veers strictly to the left} if $\phi(a)<a$. We say $a$ \emph{veers to the right} if it is fixed or if it veers strictly to the right, i.e.~if $a\leq \phi(a)$; analogously we define veering to the left.
\end{definition}
\begin{remark}[Fixed arcs, irreducibility, and primeness of links]
    Fixed arcs correspond to product discs that (up to isotopy) have as their image an embedded two-sphere whose intersection with the surface is an arc. Accordingly, we call an incompressible surface $\Sigma$  in a $3$-manifold \emph{irreducible} if every fixed isotopy class of arcs in $\mathcal{A}(\Sigma)$ is boundary parallel. Note that irreducibility only depends on the boundary: for two incompressible surfaces $\Sigma_1$ and $\Sigma_2$ in a $3$-manifold with $\partial \Sigma_1=\partial \Sigma_2$, $\Sigma_1$ is irreducible if and only if $\Sigma_2$ is.

Alternatively, one can consider the notion of \emph{prime} surfaces $\Sigma$, which satisfy that every fixed separating isotopy class of arcs in $\mathcal{A}(\Sigma)$ is boundary parallel. We abstain from doing so, as in $S^3$ (and more generally in $3$-manifolds %
without embedded non-separating two-spheres) the two notions agree.

As mentioned above, in $S^3$ being irreducible (or, equivalently, prime) for an incompressible surface $\Sigma$ is really a property of~$\partial \Sigma$. Indeed, a \emph{link} $L$---a smooth non-empty $1$-submanifold of $S^3$---is prime if and only if one (and hence every) incompressible Seifert surface $\Sigma$ of it is irreducible; compare with the definition of primeness of a link in \cref{sec:alternative}. %
We say this to illustrate our point of view, proposed in the first paragraph of the introduction, of considering properties (here primeness) of links via any choice of incompressible Seifert surface $\Sigma$  and its monodromy, in analogy of how one may consider properties of a fibered link via the monodromy of the corresponding open book.

\end{remark}

\begin{definition}
    If a surface $\Sigma$ in a $3$-manifold $M$ satisfies that for all $a\in \A(\Sigma)$ on which the monodromy $\phi$ is defined, we have $\phi(a) \geq a$, we say that $\Sigma$ is \emph{right-veering}. If furthermore no essential isotopy class of arcs is fixed (i.e. if for every arc $a\in \A(\Sigma)$ for which $\phi$ is defined, we have $\phi(a) > a$ or $a$ has a representative that is contained in $\partial\Sigma$) we say that $\Sigma$ is \emph{strictly right-veering}. The definitions of \emph{left-veering} and \emph{strictly left-veering} are analogous.
    
    More generally, for any surface $\Sigma$ and any partially defined self-map $\phi$ on $\mathcal{A}(\Sigma)$ with the property that the start point of $a$ and $\phi(a)$ are the same for all $a\in\mathrm{domain}(\phi)$, we define $\phi$ to be (strictly) right-veering and (strictly) left-veering in the same way.
\end{definition}

We note that the above definition of a surface $\Sigma\subset M$ being (strictly) right-veering, does in fact correspond to the notion of right-veering introduced for partial open books from~\cite{HKMPartial}, via the construction outlined in~\cref{rmk:phitopartialopenbook}.

\subsection{Connected sums and Murasugi sums}
For $3$-manifolds $M_1$ and $M_3$ with a chosen $3$-ball $B_i^3\subset M_i$ in each of the interiors, we write $M_i^\circ\coloneqq M_i\setminus \mathrm{int}(B_i^3)$. The connected sum of $M_1$ and~$M_2$, denoted by~$M_1\#M_2$, is defined as the result of choosing such balls and an orientation reversing diffeomorphism $F$ identifying their boundaries and using it as a gluing map to build $M_2^\circ\cup_F M_1^\circ$. The diffeomorphism type of the connected sum only depends on the choice of connected component of the balls. In particular, for connected~$M_i$ the diffeomorphism type of the connected sum is well-defined. For what follows, we fix a choice of $3$-balls $B_1^3$ and $B_2^3$ in $M_1$ and~$M_2$, respectively, and an identification $F$ of their boundary that yield an explicit~$M_1\# M_2$. 

\begin{definition}\label{def:murasugissum}
    Let $\Sigma_1$ and $\Sigma_2$ be surfaces in $3$-manifolds $M_1$ and~$M_2$. 
    Let $n\geq1$ be an integer and let $P$ be the regular $2n$-gon whose
sides are circularly labeled $s_1, \dots, s_{2n}$. 
    For~$i = 1,2$, let $f_i$ be an embedding of $P$ into $\Sigma_i$ such that, for~$j = 1, \dots , n$, $f_1(s_{2j-1}) \subset \partial \Sigma_1$, $f_2(s_{2j}) \subset \partial \Sigma_2$, and $f_1(s_{2j})\subset \Sigma_1$ and $f_2(s_{2j-1})\subset \Sigma_2$ are properly embedded arcs in the respective surfaces.
    We identify $P$ with its image in~$\Sigma_1$,~$\Sigma_2$, and~$\Sigma$, and often suppress $f_i$ in the notation.
We assume that $\Sigma_i \cap B^3_i = P \subset \partial B^3_i \subset M_i$.
    Here, we arrange that a positive normal to $P\subset M_1$ points into $B_1^3$ and a positive normal to $P\subset M_2$ points out of~$B_2^3$.
    Furthermore, we arrange that $F$ exactly yields the chosen identification of~$P\subset M_i$, that is, $f_2\circ f_1^{-1}$ is the restriction of $F$ to the polygons.
    
    With this setup, we take $M\coloneqq M_1\# M_2$ and
    $\Sigma$ is the union of the copies of $\Sigma_1$ and $\Sigma_2$ in $M$ given by the canonical inclusions of $M_i^\circ$ into~$M$. In particular, after identifying  $\Sigma_1$ with their copies in~$M$, we have $\Sigma=\Sigma_1\cup \Sigma_2$ and $\Sigma_1\cap\Sigma_2=P$. See \Cref{fig:Murasugisum} for an illustration.
    \begin{figure}[b]
    \centering
    \includegraphics[width=0.9\linewidth]{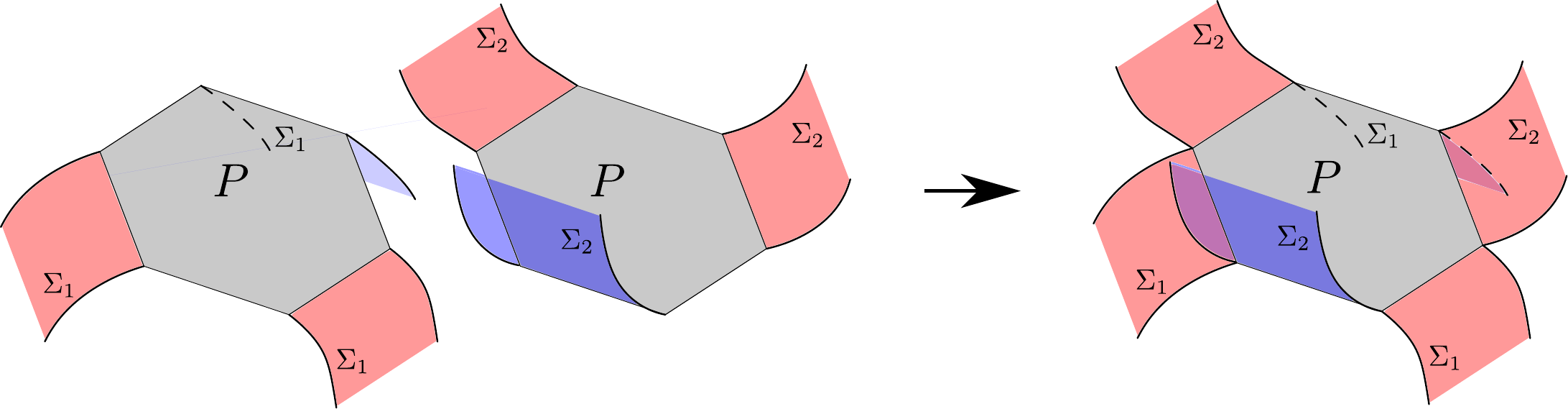}
    \caption{Illustration of the Murasugi sum $\Sigma$ of $\Sigma_1$ and~$\Sigma_2$.\newline
    Left: a neighborhood of $P$ in $M_1$ (left) and $M_2$ (right).
    \newline
    Right: A neighborhood of $P$ in~$M$.}
    \label{fig:Murasugisum}
    \end{figure}
    The pair $(M, \Sigma)$ is called the \emph{Murasugi sum of $(M_1, \Sigma_1)$ and $(M_2, \Sigma_2)$ along the summing region $P$}. For simplicity, we also say that $\Sigma$ is the Murasugi sum of $\Sigma_1$ and~$\Sigma_2$, suppressing the $3$-manifolds in the notation.

    Moreover, if for  $j = 1, \dots, n $, $f_1(s_{2j})\subset \Sigma_1$ and $f_2(s_{2j-1})\subset \Sigma_2$ are essential arcs,
    we say the Murasugi sum (or the summing region) are \emph{essential}. In another abuse of notation, we denote the sides of~$P \subset \Sigma$ by~$s_i$. Using this notation, a Murasugi sum is essential, if all the sides $s_i$ of $P\subset\Sigma$ are essential in~$\Sigma$.
\end{definition}

\section{Monodromies of Murasugi sums} \label{sec:Murasugisum}

The Murasugi sum $\Sigma$ of two surfaces is fibered if and only if both summands $\Sigma_1$ and $\Sigma_2$ are fibered. Furthermore, in this case, $\phi=\Phi_1\circ\Phi_2$, where $\Phi_i$ denotes the extension of the monodromy $\phi_i$ of $\Sigma_i$ to $\Sigma$ ($\supset \Sigma_i$) by the identity; see \cite{Stallings,Gabai, Gabaiother}. 

In the first subsection of this section, we provide a ``partial'' version of these results. That is, we characterize on which $a\in \mathcal{A}(\Sigma)$ the monodromy $\phi$ is defined, and describe the isotopy class of arcs $\phi(a)$ itself (when it is defined).
This characterization depends on the following:
for which arcs in $\mathcal{A}(\Sigma_1)$ $\phi_1$ is defined, for which arcs in $\mathcal{A}(\Sigma_2)$ $\phi_2$ is defined, as well as what the images are under $\phi_1$ and $\phi_2$.
The result can be phrased in complete analogy of the above formula $\phi=\Phi_1\circ\Phi_2$; see \Cref{rem:diskcomp}. 
While we believe this, in particular in the form given in \Cref{rem:diskcomp}, will not be of surprise to experts, we have not found this result or the strategy of proof in the literature. We also consider it the key input in this paper that allows to treat any setup where traditionally the assumption of fiberedness (in place of incompressibility) seems needed.
We will in particular use said formula to understand when an isotopy class of arcs $a$ is fixed by $\phi$: $\phi(a)=a$ if and only if $\Phi_2(a)=\Phi_1^{-1}(a)$.

In a second subsection, we give necessary conditions on the summands of a Murasugi sum that imply that having a fixed arc $a$ (i.e~$\phi(a)=a$) the restriction to the individual Murasugi summands is also fixed, which amounts $\Phi_2(a)=\Phi_1^{-1}(a)=a$. We note that there are many cases of Murasugi sums where this is not so! Consider for example the list of examples of this phenomenon, even in the case of fibered knots and links in $S^3$, provided in~\cite{FLOR}. As a light start to this subsection, we will discuss the case of connected sums (i.e.~when the plumbing region is a bigon), where it is immediate that $\Phi_2(a)=\Phi_1^{-1}(a)=a$, and how this yields a different perspective on the uniqueness of the prime decomposition of links.

Based on our criteria, we will see in \Cref{sec:trees} a general setup (tree-guided Murasugi summing), where fixed arcs in the Murasugi sum, restrict to fixed arcs of all the Murasugi summands, and it is this setup that, in the case of incompressible Seifert surfaces of links in $S^3$, will allow us to establish the visual primeness results described in the introduction.

\subsection{Description of the monodromy of a Murasugi sum in terms of the monodromy of its Murasugi summands}

\begin{lemma} \label{lemma:diskcomposition}Let $\Sigma_1\subset M_1,\Sigma_2\subset M_2$ be incompressible surfaces in $3$-manifolds $M_1$ and $M_2$. Let $\Sigma$ be a Murasugi sum of $\Sigma_1$ and~$\Sigma_2$. Let $a$ and $a'$ be in $\mathcal{A}(\Sigma)$. Then the following are equivalent.
\begin{enumerate}[(I)]
    \item \label{item:A}$\phi_{\Sigma}(a)=a'$.
    \item \label{item:B} There exists $d\in\mathcal{A}(\Sigma)$ such that $\phi_{\Sigma_2}(\Sigma_2\cap a)=d\cap \Sigma_2$ and
    $\phi_{\Sigma_1}^{-1}(\Sigma_1\cap a')=d\cap \Sigma_1$.
    \end{enumerate}

\end{lemma}

\begin{remark}\label{rem:diskcomp}
    Note the following equivalent and short formulation of \Cref{lemma:diskcomposition}. For incompressible surfaces $\Sigma_1$ and $\Sigma_2$ in $3$-manifolds $M_1$ and~$M_2$, respectively, with monodromies $\phi_1$ and~$\phi_2$, respectively, the monodromy $\phi$ of their Murasugi sum along any given summing region $P$ satisfies the following composition formula
\begin{equation}\label{eq:phi=phi1phi2}
    \phi=\Phi_1\circ\Phi_2.
    \end{equation}
    Here, $\circ$ denotes composition of partial maps\footnote{For partial maps $f\colon A\to B$ and~$g\colon B\to C$, we define $g\circ f\colon A\to C$ as the partial map with
    $\mathrm{domain}(g\circ f)\coloneqq\{a\in A\mid a\in\mathrm{domain}(f), f(a)\in \mathrm{domain}(g)\}$
    and %
    $g\circ f (a)\coloneqq g(f(a))$ for all $a\in\mathrm{domain}(g\circ f)$.} and $\Phi_i$ denotes the partial self-map on $\mathcal{A}(\Sigma)$ induced from $\phi_i$ by extension by the identity. 
    We make the latter precise for $\Phi_1$ ($\Phi_2$ is defined analogously).
    Every $a\in \mathcal{A}(\Sigma)$ can be represented by composition of arcs $\gamma_l$ that are either  contained in $\mathrm{cl}(\Sigma\setminus \Sigma_1)$ (where the closure is taken in $\Sigma$) or in $\Sigma_1$.
    We pick a minimal such representative (number of arcs minimal), and consider the arc $\gamma$ given by the composition of those $\gamma_l$ that lie in $\mathrm{cl}(\Sigma\setminus \Sigma_2)$ with representatives of $\phi_1([\gamma_l])$ if for all of those $\phi_1([\gamma_l])$ exists.
    If the latter is not the case, $\Phi_1$ is not defined on $a$, otherwise $\Phi_1(a)\coloneqq [\gamma]$.
    In other words, $\Phi_1$ is the partial map with $\mathrm{domain}(\Phi_1)\coloneqq \{a\in \mathcal{A}(\Sigma)\mid a\cap\Sigma_1\subseteq \mathrm{domain}(\phi_1)\}$ such that $\Phi_1(a)\cap\Sigma_1=\phi_1(a\cap\Sigma_1)$ and $\Phi_1(a)\cap \mathrm{cl}(\Sigma\setminus \Sigma_1)=a\cap \mathrm{cl}(\Sigma\setminus \Sigma_1)$ for all $a\in\mathrm{domain}(\Phi_1)$.
    
    Note also that, by \Cref{rmk:phitopartialopenbook}, \Cref{eq:phi=phi1phi2} recovers the celebrated results about fiberedness of Murasugi sums of surfaces stated in the first paragraph of this section.
\end{remark}

To establish \cref{lemma:diskcomposition} \Cref{item:A} $\Rightarrow$ \Cref{item:B}, we will analyze how a product disk $D$ that witnesses $\phi_{\Sigma}(a)=a'$ intersects $M_2^\circ$ and $M_1^\circ$, and how the components of these intersections can be interpreted as product disks for $\Sigma_2\subset M_2$ and $\Sigma_1\subset M_1$ that witness $\phi_{\Sigma_2}(\Sigma_2\cap a)=d\cap \Sigma_2$ and $\phi_{\Sigma_1}^{-1}(\Sigma_1\cap a')=d\cap \Sigma_1$, respectively. For the converse, \Cref{item:B} $\Rightarrow$ \Cref{item:A}, we will see how a set of product discs that witness $\phi_{\Sigma_2}(\Sigma_2\cap a)=d\cap \Sigma_2$ and
    $\phi_{\Sigma_1}^{-1}(\Sigma_1\cap a')=d\cap \Sigma_1$ can be reinterpreted in $M$ such that they glue together to give a product disc $D$ for $\Sigma\subset M$ that witnesses $\phi_{\Sigma}(a)=a'$. Here are the (somewhat long) details.
\begin{proof}[Proof of \cref{lemma:diskcomposition}]
Let us first fix our setup for the Murasugi sum. Pick a disk $P'$ with boundary equal to $\partial P$ and interior disjoint from $\Sigma$, such that $P\cup P'$ is the splitting sphere of~$M$.
In particular, the 3--manifolds $M_1^{\circ}$ and $M_2^{\circ}$ have $P\cup P'$ as a boundary component.
Denote by $\Sigma_i^{\circ}$ the closure of $\Sigma_i \setminus P$.
We find that 
$\Sigma'\coloneqq(\Sigma\setminus P)\cup P'$ is (after smoothing) a surface
that is diffeomorphic to~$\Sigma$
 via a diffeomorphism from $\Sigma$ to $\Sigma'$ that is the identity away from~$P$. We write $\Sigma_i'$ for the image of~$\Sigma_i$ under this diffeomorphism.

    \textbf{\Cref{item:A} $\Rightarrow$ \Cref{item:B}:}
    Let $a$ and $a'$ be in $\mathcal{A}(\Sigma)$ with $\phi(a) = a'$, and
let $\gamma$ and $\gamma'$ be arcs in $\Sigma$ representing $a$ and $a'$, respectively.
We pick these representatives such that they intersect $\partial P$ minimally.
By definition of $\phi$, there exists a product disk witnessing $\gamma$ going to~$\gamma'$,
which may be made transverse to $P'$ by a small isotopy.
Among all such product disks (which need not all be isotopic),
we pick a product disk $D$ such that the number of components of $D\cap P'$ is minimal.

It follows that $D\cap P'$ has no circle components. Otherwise, consider the innermost circle component in $P'$. It bounds disks $E \subset D$ and $E' \subset P'$, with $(E')^{\circ} \cap D = \varnothing$.
Then $(D \setminus E) \cup E'$ is a product disk such that, after a small isotopy,
its intersection with $P'$ has fewer components than $D \cap P'$. This contradicts the minimality assumption above.

We consider the connected components $D_1,\dots,D_j,\dots, D_N$ of $D\cap M_1^\circ$ and~$D\cap M_2^\circ$. These connected components are disks (in fact, as a smooth surface with corners each is a $2n$-gon for some~$n$, with sides alternating between subsets of $P'$ and of $\gamma\cup \gamma'$), and, in the rest of the proof, we simply refer to them as ``disks'' or as $1$-disks or $2$-disks to indicate they are components of $D\cap M_1^\circ$ and $D\cap M_2^\circ$, respectively.
See \Cref{fig:Schematic_Disk} for an example of a disk $D$ for which~$N=5$.
    \begin{figure}[ht]
    \centering
    \includegraphics[width=0.6\linewidth]{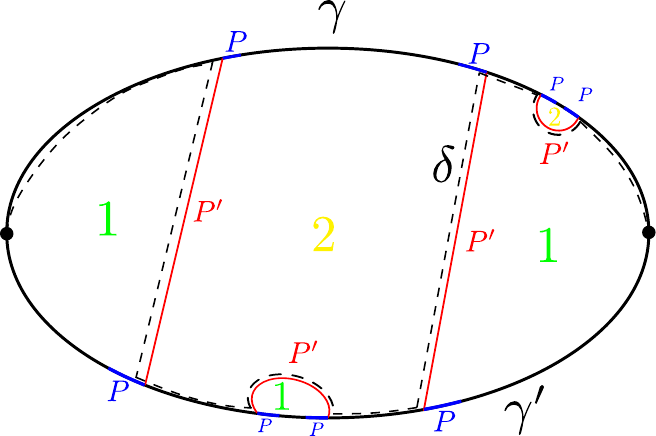}
     \caption{A schematic picture of the disk~$D$, together with its intersections with $M_1^\circ$ and~$M_2^\circ$, where 1 (green) indicates a 1-disk and 2 (yellow) indicates a 2-disk. In this case, the arc $\delta$ is obtained as the union of the negative parts of the boundaries of the 2-disks and the positive part of the boundaries of the 1-disks (note that these coincide on~$P'$).}
\label{fig:Schematic_Disk}
    \end{figure}
In fact, this example can be taken as a guiding example through the rest of the proof.

We start by observing certain restrictions on how the disks $D_j$ can look and lie inside~$D$. For this we describe $D$ locally, near intersections of $P'$ with $\gamma$ and~$\gamma'$.
First, note that a positive normal pushes $\gamma \cap P$ into $M_2$; see \Cref{fig:Local_arc}. Therefore, a component of $\gamma \cap P$ must always be part of a $2$-disk.
    \begin{figure}[ht]
    \centering
    \includegraphics[width=0.6\linewidth]{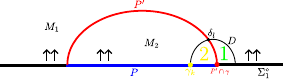}
    \caption{The disk $D$ near~$P' \cap \gamma$, divided into its intersection with $M_2^\circ$ (labeled~$2$), and $M_1^\circ$ (labeled~$1$). Here the intersection of $D$ and $P$ is contained in~$2$.}
\label{fig:Local_arc}
    \end{figure}
    Second, a negative normal pushes $\gamma' \cap P$ into~$M_1$; see \Cref{fig:Local_arc_image}. Therefore, a component of $\gamma' \cap P$ must always be part of a $1$-disk.
    
\begin{figure}[ht]
    \centering
    \includegraphics[width=0.6\linewidth]{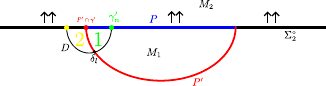}
    \caption{The disk $D$ near~$P' \cap \gamma'$, divided into its intersection with $M_1^\circ$ (labeled~$1$), and $M_2^\circ$ (labeled~$2$). Here the intersection of $D$ and $P$ is contained in~$1$.}
\label{fig:Local_arc_image}
    \end{figure} With these two local descriptions of $D$ close to the intersections of $P'$ with $\gamma$ and~$\gamma'$, and invoking minimality of the intersections of $D$ with $P'$ and $\partial D$ with $\partial P=\partial P'$, we observe the following about~$D$.
    
    \begin{claim}\label{claim:shapeofPicapD}
Let $I$ be a connected component of~$P'\cap D$. If both endpoints of $I$ lie on~$\gamma$, then $I$ union the subinterval of $\gamma$ with the same endpoints is the boundary of some $2$-disk (in fact, a bigon).
If instead both endpoints lie on~$\gamma'$, then $I$ union the subinterval of $\gamma'$ with the same endpoints is the boundary of some $1$-disk.
\end{claim}

As an illustration of \Cref{claim:shapeofPicapD}, consider $D$ given in \Cref{fig:Schematic_Disk}.
There are four connected components of~$P'\cap D$.
Two have one end point on $\gamma$ and one on~$\gamma'$.
The other two have both end points on $\gamma$ or both end points on $\gamma'$.
Observe that the latter two are the boundaries of disks $D_j$, as claimed.

\begin{proof}[Proof of \Cref{claim:shapeofPicapD}]
We consider the case of both endpoints of $I$ lying on~$\gamma$. A completely analogous argument switching 1 with 2 and $\gamma$ with~$\gamma'$, yields the second part of the claim.

Note that $I$ and $\gamma$ cut out a bigon $B$ from~$D$,
which is a union of $1$- and $2$-disks.
The statement to be proven is simply that $B$ consists of a single $2$-disk.
Let us assume toward a contradiction that this is not the case, and thus, that there exists a
1-disk $D_j$ that is disjoint from $\gamma'$.

Such a 1-disk is a polygon with half of its sides on $P'$, and half of 
them on $\gamma$. There is one side on $P'$, ``the long side", such that the 
long side together with a segment of $\gamma$ encloses $D_j$. Let us take $D_j$ 
to be innermost, in the sense that $D_j$ is the only 1-disk enclosed by 
the long side together with a segment of $\gamma$.

As discussed above, a component of $\gamma \cap P$ must always be part of a 2-disk.
So $(\partial D_j) \cap \gamma$ is disjoint from the interior of $P$,
and thus $\partial D_j \subset \Sigma_1^{\circ} \cup P' = \Sigma'_1$.
Therefore, $D_j$ is an embedded disk in $M_1^{\circ}$ with boundary on the surface $\Sigma'_1$.
However, $\Sigma'_1$ is isotopic to $\Sigma_1$ in~$M_1$, which is incompressible,
and so $\Sigma'_1$ is incompressible itself, so there's a disk in $\Sigma'_1$ with boundary $\partial D_j$. This implies that there is a disk $E$ in $\Sigma_1$ whose boundary is sent 
to the boundary of $D_j$ by the homeomorphism between $\Sigma_1$ and $\Sigma_1'$. The 
boundary of $P$ intersects $E$ in a non-zero number of chords. So there are 
at least two bigons formed by the boundary of $E$ and the boundary of $P$. 

It follows that there is at least one bigon formed by (1) $\gamma$ and the 
boundary of $P$, or (2) a $P'$ side s of $D_j$ (sent to $\Sigma_1$) and the 
boundary of $P$. See \Cref{fig:BigonCase2} for an illustration of the latter case.

    \begin{figure}[ht]
    \centering
    \includegraphics[width=0.6\linewidth]{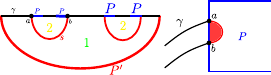}
    \caption{The bigon formed by a segment $s$ of $P'$, sent to $\Sigma_1$, and the boundary of $P$, in case (2). Note that, in particular, the points $a$ and $b$ are on the same side of $P$, and on the other side of $s$ there is a 2-disk which is a bigon whose boundary is the union of $s$ and a segment of $\gamma$.}
\label{fig:BigonCase2}
    \end{figure}

In case (1), we reach a contradiction because we had assumed that $\gamma$ and $\partial P$ intersect minimally. In case (2), we may assume, since there are two bigons, that $s$ is not the long side.

Then it follows that there is a 2-disk whose boundary consists of a $P'$
segment (the segment $s$ from above) and a segment $g$ of $\gamma$, such that 
the two endpoints of $s$ lie on the same side of the boundary of $P$. To reach the desired contradiction, we show that such a 2-disk does not exist.

Indeed, by pushing through the bigon formed by s and the boundary of $P$, 
we can isotope that 2-disk to a disk $F$, such that the boundary of $F$ 
consists of $g$ (the segment of $\gamma$ from above), and a segment on the 
boundary of $P$. We can do that isotopy without introducing any 
intersections between the interior of F and $\Sigma$. 
So now, we can apply 
the previous argument to $F$: since $F$ is embedded in $M_2$, with boundary on 
$\Sigma_2$, which is incompressible, we get a disk $G$ embedded into 
$\Sigma_2$ with the same boundary as $F$. This disk contains a bigon formed 
by a segment of $\gamma$ and the boundary of $P$ (again using the fact 
that there are at least two bigons, and at most one of them can contain 
$s$). This contradicts minimality. See \Cref{fig:Otherbigon} for an illustration.
\begin{figure}[ht]
    \centering
    \includegraphics[width=0.6\linewidth]{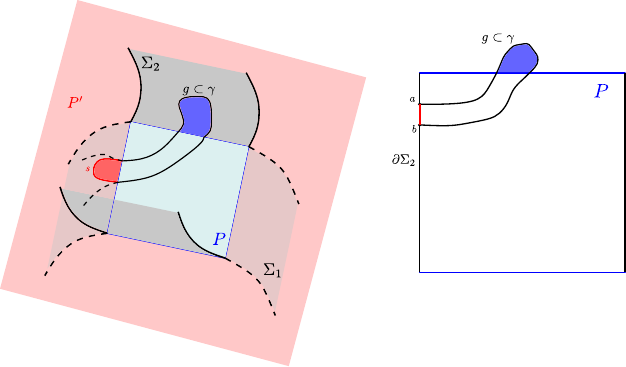}
    \caption{Using the bigon in $P'$ from before (shaded dark red), we can isotope away from $P'$ so that we get a disk contained in $\Sigma_2$. But then we must have another bigon formed by a segment of $\gamma$ and the boundary of $P$ (shaded blue) which gives the desired contradiction.}
\label{fig:Otherbigon}
    \end{figure}
\end{proof}

With the analysis of $P'\cap D$ from \Cref{claim:shapeofPicapD}, we are ready to start defining an arc~$\delta$ such that $d\coloneqq[\delta]\in\mathcal{A}$ is as desired in \Cref{item:B}.
Towards this, we define a subarc $\delta_i$ of $\delta$ for each disk $D_i$ and $\delta$ will be their union. Recall that a disk $D_i$ is a connected component of $D\cap M_1^\circ$ or~$D\cap M_1^\circ$.
We first discuss the case $D_i \subset D\cap M_2^\circ$. Note that the by \Cref{claim:shapeofPicapD}%
, $P' \cap D_i$ intersects $\gamma$ exactly twice. We denote these points $p_1$ and~$p_2$, ordered by the orientation of $\gamma$ . In particular, there is a segment $\gamma_i \subset \gamma$ going from $p_1$ to~$p_2$. Let $\delta_i$ be the segment of $\partial D_i$ that starts in $p_1$ and ends in~$p_2$, and is disjoint from $\gamma'$ in its interior. Using the identification of $\Sigma_2$ and~$\Sigma_2'$, this implies that $D_i$ can be identified with a product disk in $M_2$ that witnesses $\gamma_i$ going to~$\delta_i$. See \Cref{fig:Disk2} for an illustration.
\begin{figure}[t]
    \centering
    \includegraphics[width=0.3\linewidth]{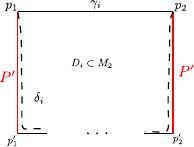}
    \caption{A schematic picture of a disk~$D_i \subset M_2$.}
\label{fig:Disk2}
    \end{figure}
    
Next, we discuss the case $D_i \subset D\cap M_1^\circ$. Let $p_1$ be the first intersection point of $P' \cap D_i$ and $\gamma$ (the order coming from the orientation of~$\gamma$), and $p_2$ the last intersection point of $P' \cap D_i$ and~$\gamma$. Now, note that by \Cref{claim:shapeofPicapD}%
, $P' \cap D_i$ intersects $\gamma'$ exactly twice, in points $p'_1$ and $p'_2$ ordered by the orientation of $\gamma'$ (here the orientation of $\gamma'$ is the one that makes it have the same starting point as~$\gamma$).
In particular, there is a segment $\gamma'_i \subset \gamma'$ going from $p'_1$ to~$p'_2$.  Let $\delta_i$ be the segment in $\partial D_i$ that starts in $p'_1$ and ends in~$p'_2$, and is disjoint from $\gamma'$ in its interior. Using the identification of $\Sigma_1$ and~$\Sigma'_1$, this implies that $D_i$ can be identified with a product disk in $M_1^\circ$ that witnesses $\delta_i$ going to~$\gamma'_i$. See \Cref{fig:Disk1} for an illustration.
\begin{figure}[t]
    \centering
    \includegraphics[width=0.3\linewidth]{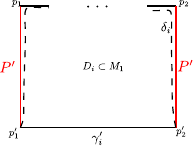}
    \caption{A schematic picture of a disk~$D_i \subset M_1$.}
\label{fig:Disk1}
    \end{figure}

 The desired $d\in\mathcal{A}(\Sigma)$ turns out to be the isotopy class of the arc $\delta\coloneqq \cup_{1}^{N}\delta_i\subset \partial (D\cup D\cap P')\subset D$, i.e.~the union of all the segments~$\delta_i$, where we make use of the identification of $\Sigma$ with~$\Sigma'$, meaning we interpret $P'$ as a subset of~$\Sigma$. We argue that $\delta$ is connected, in fact it is an arc with boundary equal to the boundary of~$\gamma$, and we orient $\delta$ as to have the same start point as~$\gamma$. Towards this, let $D_1$ denote the disk that contains the start point of $\gamma$ and $D_N$ the disk that contains the end point of~$\gamma$, and lets analyze all disks $D_i$ that are neither $D_1$ nor~$D_N$.
 
 Note that a disk $D_i \subset M_2^\circ$ which intersects both $\gamma$ and $\gamma'$ contains in its boundary exactly two subarcs given as connected components of $P'\cap D$ that have end points on both $\gamma$ and~$\gamma'$. In fact, one of these arcs has the boundary point on equal to $p_1$ and~$p_1'$, let is call it~$I_1$, and the other has the boundary point equal to $p_2$ and~$p_2'$, lets call it~$I_2$. Hence, there are two disks $D_{j} \subset M_1^\circ$ and $D_{k} \subset M_1$ that meet $D_i$ along those two subarcs, say $D_j\cap D_i=I_1$ and~$D_k\cap D_i=I_2$; see left-hand side of \Cref{fig:Localdelta}.
 \begin{figure}[ht]
    \centering
    \includegraphics[width=0.9\linewidth]{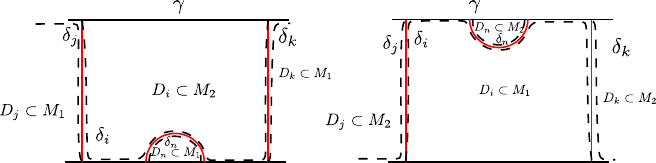}
    \caption{The segments $\delta_i$ overlap over~$P'$, and moreover join up to a connected arc. This also happens in the bubbles, where the segments $\delta_n$ and $\delta_m$ coincide with a segment of~$\delta_i$.}
\label{fig:Localdelta}
    \end{figure}
 Analogously a disk $D_i \subset M_2^\circ$ that intersects both $\gamma$ and $\gamma'$ has corresponding disks $D_{j} \subset M_1^\circ$ and~$D_{k} \subset M_1$; see right-hand side of \Cref{fig:Localdelta}. In this case, the segments $\delta_i$ and $\delta_{j}$ overlap on the segment of $P'$ between $p_1$ and~$p'_1$, and $\delta_i$ and $\delta_{k}$ overlap on the segment of $P'$ between $p_2$ and~$p'_2$. Moreover, by \Cref{claim:shapeofPicapD} the only disks that can be between these two segments are bubbles: disks in $M_1^\circ$ with an arc in $ \gamma$ on the positive side, and an arc in $P'$ on the negative side (if~$D_i \subset M_2^\circ$); and disks in $M_2^\circ$ with an arc in $P'$ on the positive side, and an arc in $\gamma'$ on the negative side (if~$D_i \subset M_1^\circ$); see \Cref{fig:Localdelta}. For the disks $D_1$ and~$D_N$, which also both intersect $\gamma$ and~$\gamma'$, the situation is similar: If~$N=1$, we have only one subarc $I$ of the boundary that is a connected component of $P'\cap D$ that intersects both $\gamma$ and~$\gamma'$. In case of $D_N$ the boundary of $I$ is $p_1$ and~$p_1'$, and we have a unique disk $D_j$ with~$D_j\cap D_N=I$. In case of $D_1$ the boundary of $I$ is $p_2$ and~$p_2'$, and we have a unique disk $D_j$ with~$D_j\cap D_1=I$. If~$1=N$, we have~$D_1=D_N=D$.
 
 If a disk $D_n$ is disjoint from $\gamma$ or~$\gamma'$, then by \Cref{claim:shapeofPicapD} it is a bigon, so it is one of the bubbles discussed above. Then, there is a unique disk $D_i$ that it is adjacent to---instead of two, but as before the segment $\delta_n$ obtained from it coincides with a segment of~$\delta_i$. See again \Cref{fig:Localdelta} for an illustration of this.
 This ensures that $\delta$ is indeed a well-defined arc as claimed, and moreover that the disks in $M_2^\circ$ witness $\phi_2(\Sigma_2 \cap a) = \Sigma_2 \cap d$ and the disks in $M_1^\circ$ witness $\phi_1^{-1}(\Sigma_1 \cap a') = \Sigma_1 \cap a$.

    \textbf{\Cref{item:B} $\Rightarrow$ \Cref{item:A}:} We turn to the converse. That is, we assume that there exists $d\in\mathcal{A}(\Sigma)$ such that $\phi_2(\Sigma_2\cap a)=d\cap \Sigma_2$ and
    $\phi_1^{-1}(\Sigma_1\cap a')=d\cap \Sigma_1$, and pick arcs $\gamma, \gamma'$, and $\delta$ that are representatives of $a,a'$, and $d$, respectively, such that their intersections with $P\subset \Sigma_1$, $P\subset \Sigma_2$, $\Sigma_1$, and $\Sigma_2$ are minimal.
    Let $D_1,\dots,D_j,\dots, D_N$ be the product disks that witness $\phi_2(\Sigma_2\cap a)=d\cap \Sigma_2$ and
    $\phi_1^{-1}(\Sigma_1\cap a')=d\cap \Sigma_1$. We isotope them such that each of them is contained in $M_1^\circ$ or $M_2^\circ$ as follows. For a disk $D_j$ in $M_2$ witnessing $\phi_2(\Sigma_2\cap a)=d\cap \Sigma_2$, that means a product disk $D_j\subseteq M_2$ for $\Sigma_2$ bringing a segment of $\Sigma_2\cap \gamma$ to a segment of $\delta\cap \Sigma_2$,
    we use the identification of $\Sigma_2$ with $\Sigma_2'$ to see the segment of $\delta\cap \Sigma_2$ as a subset of~$\Sigma_2'$. Similarly, for a disk $D_j$ in $M_1$ witnessing $\phi_1^{-1}(\Sigma_1\cap a')=d\cap \Sigma_1$, that is a product disk $D_j\subseteq M_1$ for $\overline{\Sigma_1}$ bringing a segment of $\delta\cap \Sigma_1$
    to a segment of $\Sigma_1\cap \gamma'$, we use the identification of $\Sigma_1$ with $\Sigma_1'$ to see the segment of $\delta\cap \Sigma_1$ as a subset of~$\Sigma_1'$. By abuse of notation, we call the resulting disk in $M_2^\circ$ and $M_1^\circ$, respectively, also $D_j$ and only consider those $D_j$ from here on out.

    The idea of the proof is to observe that the union of all the $D_j$ (where the union happens along all the segments of their boundary that lie in~$P'$) results in a product disk $D$ that witnesses~$\gamma$ going to $\gamma'$; see \Cref{fig:Gluing_disks'} for an example with the disks from \Cref{fig:Schematic_Disk}.
 \begin{figure}[ht]
    \centering
    \includegraphics[width=0.8\linewidth]{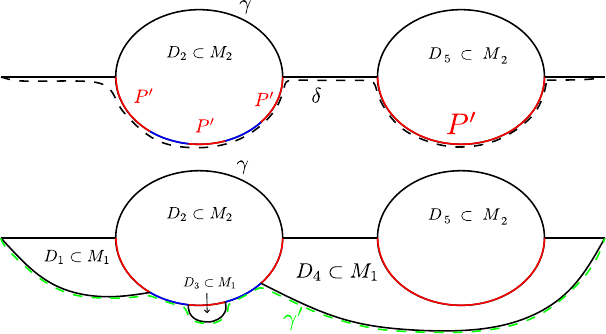}
    \caption{On the top, an arc $\gamma$ together with two disks $D_2, D_5 \subset M_2$, going from $\gamma$ to~$\delta$. On the bottom, the disks $D_1, D_3, D_4 \subset M_1$ going from $\delta$ to~$\gamma'$. Gluing all of the disks together yields the product disk $D$ witnessing~$\phi(\gamma) = \gamma'$.}
\label{fig:Gluing_disks'}
    \end{figure}
    We introduce notation towards making this precise. 
    
    As before, we call a $D_j$ an $i$ disk, if it is contained in~$M_i^\circ$. We let $D_1$ be the product disk that contains the start point of $\delta$ and let $D_N$ be the product disk that contains the end point of~$\delta$. Both of these intersect both $\gamma$ and~$\gamma'$. 

    We consider $D_i$ different from $D_1$ and~$D_N$, and distinguish two different cases. First, the case where $D_i$ intersects both $\gamma$ and~$\gamma'$. Second, the case where $D_i$ only intersects one of them. Note that by definition $D_i$ is a bigon whose boundary is the union of a segment of $\gamma$ (or~$\delta$) and a segment of $\delta$ (or~$\gamma'$). Moreover, the segment that is not on $\gamma$ (or~$\gamma'$) must lie on~$P'$.

    \begin{claim} \label{claim:gluing_segment}
        If a segment is on negative side of the boundary of a $2$-disk and on the positive side of the boundary of a $1$-disk, then it lies entirely on~$P'$.
    \end{claim}

    \begin{proof} [Proof of \Cref{claim:gluing_segment}] 
        If a segment belongs to a $1$-disk and a $2$-disk, then it needs to be on the separating sphere dividing $M_1$ and~$M_2$, which is~$P \cup P'$. Then, on $P$ the orientation would be wrong, because on $P$ the negative side lies on~$M_1$, not~$M_2$, so the segment must lie on~$P'$.
    \end{proof}

    Suppose that $D_i$ intersects both $\gamma$ and~$\gamma'$. If~$D_i \subset M_2$, then one segment of the boundary of the bigon lies on~$\gamma$, and the other segment lies on $\delta$ (and this contains parts of~$\gamma'$). The segments on $P'$ at the endpoints connect $\gamma $ to~$\gamma'$. Then, these segments also belong to disks $D_j$ and $D_k$ in~$M_1$, which also intersect both $\gamma$ and~$\gamma'$.
     \begin{figure}[b]
    \centering
    \includegraphics[width=0.5\linewidth]{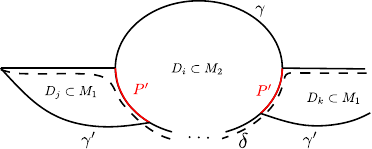}
    \caption{A disk $D_i \subset M_2$ that intersects both $\gamma$ and~$\gamma'$, glued to disks $D_j, D_k \subset M_1$. Both of these disks also intersect $\gamma$ and $\gamma'$.}
\label{fig:gamma_to_gamma'}
    \end{figure}
    The situation with $D_i \subset M_1$ is analogous.

    Now, for $D_1$ and~$D_n$, the situation is similar, as they intersect both $\gamma$ and~$\gamma'$, except now in both cases there is a unique segment of $P'$ connecting $\gamma$ and~$\gamma'$. On the other side of this segment, there is another disk $D_j$ which is of the form discussed above.
    
    Finally, suppose that $D_i$ only intersects $\gamma$ and not~$\gamma'$. Then, it is a bigon with one side on $\gamma$ and another one lying entirely on~$P'$. Then we must have that~$D_i \subset M_2$, since $1$-disks move segments from $\delta$ to $\gamma'$ so they cannot be disjoint from~$\gamma'$. Moreover, there is a unique disk $D_j \subset M_1$ on the other side of the segment lying on~$P'$, and this disk intersects both $\gamma$ and~$\gamma'$. See \Cref{fig:Bubble} for an illustration.

    \begin{figure}[b]
    \centering
    \includegraphics[width=0.3\linewidth]{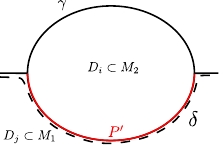}
    \caption{A disk that intersects $\gamma$ but not~$\gamma'$.}
\label{fig:Bubble}
    \end{figure}
    
    Similarly, if $D_i$ intersects $\gamma'$ and not~$\gamma$, then~$D_i \subset M_1$, and the segment of $\partial D_i$ that does not lie on $\gamma$ lies on~$P'$. Moreover, there is a unique disk $D_j \subset M_2$ on the other side of the segment lying on~$P'$, and this disk intersects both $\gamma$ and~$\gamma'$.

    Then, gluing all disks together yields a product disk witnessing $\phi(\gamma) = \gamma'$.
\end{proof}

We obtain the following immediate corollary for fixed arcs in Murasugi sums.

\begin{corollary}\label{cor:fixedarcinMS}
    Let $\Sigma$ be the Murasugi sum of incompressible surfaces $\Sigma_1$ and~$\Sigma_2$. Let $a$ be in~$\mathcal{A}(\Sigma)$. Then the following are equivalent.

\begin{itemize}
    \item $a$ is fixed.
    \item There exists $d\in\mathcal{A}(\Sigma)$ such that $\phi_1(\Sigma_1\cap a)=d\cap \Sigma_1$ and
    $\phi_2^{-1}(\Sigma_2\cap a)=d\cap \Sigma_2$.
\end{itemize}

\end{corollary}

\begin{proof}
    This is an immediate application of \Cref{lemma:diskcomposition}. %
\end{proof}

\begin{remark}\label{rem:fixedarcs}
In the language of \Cref{rem:diskcomp},  \Cref{cor:fixedarcinMS} is simply the statement that $a\in\mathcal{A}(\Sigma)$ is a fixed arc (i.e.~$\phi(a)=a$) if and only if $\Phi_2(a)=\Phi_1^{-1}(a)$, which is immediate from~\eqref{eq:phi=phi1phi2}. Note that if $a$ is a fixed arc, $\Phi_2(a)=\Phi_1^{-1}(a)$ in general is different from $a$.

The rest of the section is devoted to necessary conditions that allow to conclude that if $a$ is fixed, then $\Phi_2(a)=\Phi_1^{-1}(a)=a$.
\end{remark}

\subsection{Necessary conditions for fixed arcs in Murasugi sums to restrict to fixed arcs in the summands}

We start with the case where the Murasugi sum is trivial, i.e. when the sum is directly a connected sum of the boundary links, and the boundary connected sum of the surfaces. Here the summing region $P$ is a bigon. This means that $\Sigma_i \setminus P$ is obtained from $\Sigma_i$ by cutting out a boundary parallel disk, and thus for any arc $a\subset \Sigma$, there is a natural identification of each component of $a\cap \Sigma_i$ with each component of $a \cap (\Sigma_i \setminus P)$. But by \Cref{cor:fixedarcinMS} we have that $\gamma \cap (\Sigma \setminus P) = \delta \cap (\Sigma \setminus P)$ where $\gamma $ is a representative of $a$ and $\delta$ is a representative of $\Phi_2(a)=\Phi_1^{-1}(a)$. Hence we immediately get that $a \cap \Sigma_i$ is a collection of fixed arcs in $\Sigma_i$, that is, we have proved the following proposition.

\begin{proposition} \label{prop:connectsumcrit}
    Let $\Sigma$ be the boundary connected sum of incompressible surfaces $\Sigma_1$ and $\Sigma_2$. If there exists an arc $a$ such that $\phi(a)=a$, then $\Phi_1^{-1}(a) = \Phi_2(a) = a$.\qed
\end{proposition}

\begin{remark}
    As a fun aside, it is not hard to see that \Cref{prop:connectsumcrit} recovers uniqueness of prime decomposition of links in $S^3$ (or null-homologous links in any irreducible $3$-manifold $M$), where the decomposing spheres are given by the product disks that correspond to fixed arcs. (Hint: while, in general, the choice of collection of fixed arcs exhibiting the prime decomposition is not unique, \Cref{prop:connectsumcrit} shows that the resulting irreducible surfaces from any two such choices are isotopic.)
\end{remark}

Now, we will see that we can obtain similar results even if the Murasugi sum is not immediately a connected sum, provided we have some information of the images of arcs in the original surfaces. 

Let $P\subset \Sigma_i$ be a plumbing region (an alternatingly embedded $2n$-gon) in  a surface $\Sigma_i$ in $3$-manifold $M_i$, and $a,b\in \mathcal{A}(\Sigma_i)$ with start point $p$ on $\partial \Sigma_i \cap P=\partial\Sigma_i\cap \partial P$. Below we define $b$ being to the right (and strictly to the right) of $a$. The idea is to consider representatives that intersect $P$ minimally, taking the segments that start at $p$ and consider them up to isotopy in $P$, where isotopies have to fix $\partial P\cap \partial \Sigma_i$, and then see whether they are to the right or strictly to the right; see \Cref{fig:a<Pb} for illustration.
\begin{figure}[ht]
    \centering
    \includegraphics[width=0.4\linewidth]{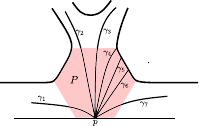}
    \caption{Example of a plumbing region $P$ (red), and arcs $\gamma_1,\dots,\gamma_7$ starting at the point  $p\in\partial \Sigma$ such that $[\gamma_1]<[\gamma_2]<[\gamma_3]<[\gamma_4]<[\gamma_5]<[\gamma_6]<[\gamma_7]$ and $[\gamma_1]<_P[\gamma_2]=_P[\gamma_3]<_P[\gamma_4]<_P[\gamma_5]<_P[\gamma_6]<_P[\gamma_7]$.}
\label{fig:a<Pb}
    \end{figure}

We say $b$ is strictly to the right of $a$ in $P$, written $a<_P b$ if the elements of $b\cap P$ with starting point $p$ are strictly to the right of the elements of $a\cap P$ with starting point $p$. In other words, if $a=[\gamma]$ and $b=[\delta]$ such that $\gamma$ and $\delta$ intersect $P$ minimally, and $s_\gamma$ and $s_\delta$ denote the arcs given as the components of $P\cap \gamma$ and $P\cap \delta$ that start at $p$, then $s_\delta$ ends to the right of $s_\gamma$ and if neither ends in $\partial P\cap \partial \Sigma$, then they end on different sides of the polygon $P$.

We say $b$ is to the right of $a$ in $P$, written $a\leq_P b$, if $a<_Pb$ or if $a=_Pb$, where the latter is defined to mean that elements of $b\cap P$ with starting point $p$ are equal to the elements of $a\cap P$ with starting point $p$. In other words, if $a=[\gamma]$ and $b=[\delta]$ such that $\gamma$ and $\delta$ intersect $P$ minimally, and $s_\gamma$ and $s_\delta$ denote the arcs given as the components of $P\cap \gamma$ and $P\cap \delta$ that start at $p$, then $a=_Pb$ means that either $s_a$ and $s_b$ end in the same point on $\partial P$ or they end in the same component of $\partial P\setminus \partial \Sigma$.

Note that $a\leq b$ does imply $a\leq_P b$, but $a<b$, in general, does not imply $a<_Pb$, see e.g.~$a=[\gamma_2]$ and $b=[\gamma_3]$ in \Cref{fig:a<Pb}.

With this setup, it is natural to make the following definition.

\begin{definition}
    Let $P\subset \Sigma_i$ be a plumbing region (an alternatingly embedded $2n$-gon) in  a surface $\Sigma_i$ in $3$-manifold $M_i$. Suppose that for all $a\in \mathcal{A}(\Sigma_i)$, with start point $p$ on $\partial \Sigma_i \cap P=\partial\Sigma_i\cap \partial P$, and such that $\phi(a)$ is defined, $a \leq_P \phi(a)$. Then, we say that $\phi_i$ is \emph{right-veering in $P$}.
\end{definition}

\begin{proposition} [Primeness Criterion]\label{crit}
Let $\Sigma_1$ and $\Sigma_2$ be incompressible surfaces in $3$-manifolds $M_1$ and~$M_2$. Denote by $\phi_2$ the monodromy of $\Sigma_2$ and by $\phi_1^{-1}$ the monodromy of~$\overline{\Sigma_1}$.
Let $\Sigma$ be a surface in $M_1 \# M_2$ obtained by a Murasugi sum of $\Sigma_1$ and $\Sigma_2$ along some summing region~$P$. Suppose that $\phi_2$ and $\phi_1^{-1}$ are right-veering in~$P$.

Then, if $a\in\mathcal{A}(\Sigma)$ is fixed, $a \cap \Sigma_i$ are a (possibly empty for one but not for both~$i$) collection of fixed isotopy classes of arcs in $\mathcal{A}(\Sigma_i)$, for~$i = 1,2$.

In particular, if the Murasugi sum is essential and neither $\Sigma_1$ nor $\Sigma_2$ have an essential fixed isotopy class of arcs, then $\Sigma$ has no essential fixed isotopy class of arcs.
\end{proposition}
\begin{remark}\label{rem:crit}
Writing $\Phi_i$ for the partial self-maps of $\mathcal{A}(\Sigma)$ as defined in~\cref{rem:diskcomp}, the conclusion of \cref{crit} can be rephrased as follows: if $a\in\mathcal{A}(\Sigma)$ is fixed by $\phi$, then $a$
is fixed by $\Phi_1$ and $\Phi_2$. \end{remark}

For the proof of \Cref{crit}, we make use of the following lemma about chords in disks, where a chord is an unoriented straight line between two points on the boundary of the disk. Note that we choose our disks (and polygons), to be convex.
\begin{lemma} [{\cite[Lemma 3.7]{FLOR}}] \label{lemma:lemma3}
    Let $D$ be a disk, $A$ and $B$ finite sets of chords with $\alpha \cap \alpha' = \beta \cap \beta' = \varnothing$ for $\alpha\neq \alpha' \in A$, $\beta\neq \beta' \in B$, and such that $\bigcup_{\alpha \in A} \partial \alpha = \bigcup_{\beta \in B} \partial b$. 
    Then either~$A = B$, or there exists $\alpha_1, \alpha_2 \in A$, $\beta_1, \beta_2 \in B$ with $\alpha_i \cap \beta_i = \{ p_i \} \in \partial D$, and seen from $p_1$, $\beta_1$ is to the right of~$\alpha_1$, and seen from~$p_2$, $\beta_2$ is to the left of~$\alpha_2$. 

    Moreover, if $D$ is an $n$-gon, and for each chord~$\gamma\in A\cup B$, $\gamma$ has endpoints on different sides, we can find chords as above with the additional condition that for $i = 1,2$, the other endpoint of $\alpha_i$ is on a different side than the other endpoint of~$\beta_i$.
\end{lemma}

\begin{proof}[Proof of \Cref{crit}]
Assume that there exists an $a\in\mathcal{A}(\Sigma)$ with $\phi(a)=a$. W.l.o.g.~the endpoints of $a$ are not on $\partial \Sigma_1\cap \partial \Sigma_2$ (if not, isotope them in the boundary, do the argument, and isotope back).

By \Cref{cor:fixedarcinMS}, there exists $d\in\mathcal{A}(\Sigma)$ such that $\phi_1^{-1}(a\cap \Sigma_1)=d\cap\Sigma_1$ and $d\cap\Sigma_2=\phi_2(a\cap \Sigma_2)$. In terms of \Cref{eq:phi=phi1phi2}, $\Phi_2(a)=d=\Phi_1^{-1}(a)$. Hence, we aim to show $a=d$, compare \cref{rem:crit}.
We note that, since the minimal number of intersections of each side of $P$ with $a$ and $d$ are equal, $a=d$ is equivalent to $a\cap \mathrm{cl}(\Sigma\setminus \Sigma_2)=d\cap \mathrm{cl}(\Sigma\setminus \Sigma_2)$,
$a\cap \mathrm{cl}({\Sigma\setminus \Sigma_1})=d\cap \mathrm{cl}({\Sigma\setminus \Sigma_1})$, $a\cap P=d\cap P$. The former two hold; indeed,
\[a\cap \mathrm{cl}(\Sigma\setminus \Sigma_2)\overset{\text{Def. }\Phi_2}{=}\Phi_2(a)\cap\mathrm{cl}(\Sigma\setminus \Sigma_2)\overset{\Phi_2(a)=d}{=}d\cap \mathrm{cl}(\Sigma\setminus \Sigma_2),\]
and analogously for $a\cap \mathrm{cl}({\Sigma\setminus \Sigma_1})=d\cap \mathrm{cl}({\Sigma\setminus \Sigma_1})$;
hence we have to show that $d\cap P=a\cap P$.

If $d\cap P$ is empty, then either $\varnothing=d\cap \Sigma_1$ (hence $a\cap\Sigma_1=\varnothing$ and in particular $a\cap P=\varnothing$) or $\varnothing=d\cap \Sigma_2$ (hence $a\cap\Sigma_2=\varnothing$ in particular $a\cap P=\varnothing$).
Therefore, we consider the case where $d\cap P$ is non-empty. %

To see $a\cap P=d\cap P$, we pick representatives of $a$ and $d$, let us denote them by $\gamma$ and $\delta$, respectively, and choose them to intersect the boundary of $P$ minimally and in the same points. 

Assume towards a contradiction that $a\cap P\neq d\cap P$. Note that this means that, thinking of $P$ as a polygon in~$\mathbb{R}^2$, the following two set of chords of $P$ differ: the set of chords $A$ given by straightening the components of $\gamma \cap P$ and the set of chords $B$ given by straightening the components of~$\delta \cap P$. 
Observe that $A$ and $B$ satisfy the hypotheses of \Cref{lemma:lemma3}, including the $n$-gon condition because of minimal intersection of $\gamma$ and $\delta$ with the sides of~$P$. Hence, \Cref{lemma:lemma3} yields chords~$\alpha\in A$, $\beta \in B$ such that~$\alpha \cap \beta = \{p \}$, $\alpha > \beta$ as seen from~$p$, and the other endpoints of $\alpha$ and $\beta$ (call them $p_\alpha$ and~$p_\beta$) lie on different sides of~$P$, that is, we get $\alpha >_P \beta$. 
Let $i\in\{1,2\}$ such that~$p \in \partial \Sigma_i$.
Let $\alpha^{\mathrm{ex}}$ be the unique subarc of $\gamma$ that starts at $p$ and is a proper arc in~$\Sigma_i$.
($\alpha^{\mathrm{ex}}$ extends the arc in $P$ corresponding to~$\alpha$, and may in fact just be equal to it, namely if~$p_\alpha \in \partial\Sigma_i$). Similarly extend $\beta$ to~$\beta^{\mathrm{ex}}$. Observe that we get that $[\alpha^{\mathrm{ex}}]>_P [\beta^{\mathrm{ex}} ]$.
    
Then $[\beta^{\mathrm{ex}}]= \phi_2([\alpha^{\mathrm{ex}}])$ or $[\beta^{\mathrm{ex}}]= \phi_1^{-1}([\alpha^{\mathrm{ex}}])$ if $i = 2$ or~$i = 1$, respectively. But since 
$[\alpha^{\mathrm{ex}}] >_P [\beta^{\mathrm{ex}}]$ and both $\phi_2$ and $\phi_1^{-1}$ are right-veering in $P$, we get a contradiction.

For the `in particular', we note that an essential $a\in\mathcal{A}(\Sigma)$ is such that there exists $a' \in a \cap \Sigma_i$ that is essential in $\Sigma_i$ for $i = 1$ or $2$. By the above argument, if $a$ is fixed in $\Sigma$, $a'$ would have to be fixed in $\Sigma_i$, yielding a contradiction.
\end{proof}

\begin{proposition}\label{prop:rightright}
    Let $\Sigma_1$ and $\Sigma_2$ be incompressible right-veering surfaces in $3$-manifolds $M_1$ and~$M_2$, and let $\Sigma$ be a surface in $M_1 \# M_2$ obtained by an essential Murasugi sum of $\Sigma_1$ and $\Sigma_2$ along some summing region~$P$. Then $\Sigma$ is right-veering.
    
    If, furthermore, both $\Sigma_i$ are strictly right-veering, then so is $\Sigma$; in particular, $\Sigma$ has no essential fixed isotopy classes of arcs.
    
    The same statements hold for left-veering in place of right-veering.
\end{proposition}
\begin{proof}
We note that if $\phi_i$ are right-veering partially defined self-maps on $\mathcal{A}(\Sigma_i)$, then so are the two partial self-maps $\Phi_i$ on $\mathcal{A}(\Sigma)$.

The key is to observe that composition of two partial self-maps $f$ and $g$ on $\mathcal{A}(\Sigma)$ that are right-veering is again right veering.
Indeed, for every $a$ such that $g(f(a))$ is defined, i.e.~$f(a)$ and $g(b)$ with $b=f(a)$ are defined, we have $f(a)\geq a$ and $g(f(a))\geq f(a)$ and the statement follows from $\geq$ being a total order on elements in $\mathcal{A}(\Sigma)$ with a given fixed start point $p\in\partial\Sigma$.

Both $\Phi_i$ are right-veering, hence, by the last paragraph, so is $\phi\overset{\text{\eqref{eq:phi=phi1phi2}}}=\Phi_1\circ \Phi_2$.

The furthermore follows from the fact that composition of two partial self-maps that are both strictly right-veering is strictly right-veering, which follows as above, replacing one of $\geq$ with $>$. Note that the sum being essential guarantees that if $\gamma \in \mathcal{A}(\Sigma)$ is an essential isotopy class of arcs, then either $\gamma \cap \Sigma_1$ or $\gamma \cap \Sigma_2$ contains an essential isotopy class of arcs (and thus strictly right-veering). This implies that we do not get fixed arcs in this case.
\end{proof}

\section{Tree-guided Murasugi summing%
}\label{sec:trees}
\subsection{Trees of surfaces}

We define a construction of iterative Murasugi summings of %
surfaces in $3$-manifolds, encoded by a tree with vertices surfaces (in a $3$-manifold) and edges indicating how they are summed together. We focus on the case of each individual surface being incompressible and either left- or right-veering and we show that essential fixed arcs, if they exist, arise from fixed arcs in the individual surfaces. 

Let $T$ be an oriented tree, where the vertex set $V$ consists of pairs $(M,\Sigma)$, where $\Sigma$ is a surface in a $3$-manifold $M$, and each edge, an ordered pair $e=\left(v_1=(M_1,\Sigma_1),v_2=(M_2,\Sigma_2)\right)$, is labeled by a summing region $P$ for $v_1$ and~$v_2$, i.e.~embeddings of $f_i\colon P\to \Sigma_i$ as described in \Cref{def:murasugissum}.
We further assume that all summing regions are disjoint; meaning for every pair of distinct edges $e$ and $e'$ their labels $f_1\colon P\to \Sigma_1$ and $f_2\colon P\to \Sigma_2$ for $e$ and $f_1'\colon P'\to \Sigma_1'$ and $f_2\colon P'\to \Sigma_2'$ satisfy
\[f_1(P)\cap f_1'(P')=f_2(P)\cap f_1'(P')=f_1(P)\cap f_2'(P')=f_2(P)\cap f_2'(P')=\varnothing,\] which of course is only a relevant condition when $e$ and $e'$ feature a common vertex (hence two of the surfaces  $\Sigma_1,  \Sigma_1',  \Sigma_2, \Sigma_2'$ are the same).

We call such a labeled tree a \emph{tree of surfaces}, if all surfaces are incompressible, we call it a \emph{tree of incompressible surfaces}. Note that by Gabai \cite{Gabai, Gabaiother} the Murasugi sum of incompressible surfaces is incompressible, so when we iterate this operation we can guarantee that each summand is incompressible, and we can apply our results.

We write $M_T$ and $\Sigma_T$ for the $3$-manifold and the surface resulting from performing Murasugi sums as guided by the edges and their labels. The disjointness condition guarantees uniqueness (up to canonical isotopies) of $M_T$ and $\Sigma_T$. 

\subsection{Tree-guided Murasugi sums of right-veering surfaces is right veering}

We consider the case that all surfaces in a tree are right-veering.

\begin{lemma}\label{lem:treeofrightveeringsurfaces=>rightveeringsurface}
Let $T$ be a tree of incompressible surfaces such that every vertex consists of a right-veering surface. Then $\Sigma_T\subset M_T$ is right-veering. Furthermore, if all surfaces are strictly right-veering, then so is $M_T$.
\end{lemma}
\begin{proof}
 By induction, it suffices to show that the Murasugi sum $\Sigma\subset M$ of two right-veering surfaces $\Sigma_1\subset M_1$ and $\Sigma_2\subset M_2$ is right-veering. Hence this is a consequence of \Cref{prop:rightright}.
 \end{proof}
\subsection{Fixed arcs in tree-guided Murasugi sums of veering surfaces}

We consider the case of a tree of incompressible surfaces, where each edge consists of an incompressible surface that is either left or right-veering.

\begin{theorem}\label{thm:treeofveeringsurfaces}
Let $T$ be a tree of incompressible surfaces such that every vertex $v$ consists of a surface $\Sigma_v\subset M_v$ that is right-veering or left-veering. 

If $a\in \mathcal{A}(\Sigma_T)$ is a fixed isotopy class of arcs then for every vertex $v$, $a \cap \Sigma_v$ is (a possibly empty) collection of fixed isotopy classes of arcs in $\mathcal{A}(\Sigma_v)$.

In particular, if for every vertex $v$, $\Sigma_v$ is strictly right-veering or strictly left-veering and every summing region is essential, then $\Sigma_T$ is irreducible.
\end{theorem}
\begin{proof}
Out of $T$, we build a graph $A$ by collapsing each maximal subtree $T_\mathrm{max}$ that contains only right-veering or only left-veering surfaces to a single vertex consisting of the pair $(\Sigma_{T_\mathrm{max}}, M_{T_\mathrm{max}})$ and edges are inherited from $T$ (every edge between vertices in different maximal subtrees as above, becomes an edge between the corresponding new vertices), and are labeled with the natural summing region induced by the canonical inclusion of $(M_v,\Sigma_v)$ in $(\Sigma_{T_\mathrm{max}}, M_{T_\mathrm{max}})$ for all vertices $v$ in $T_\mathrm{max}$.

In conclusion, we have a tree of incompressible surfaces $A$, where each edge $e=(v_1,v_2)$ is between oppositely veering surfaces. More precisely, at least one of the following two is true: firstly, $\Sigma_{v_1}$ is right-veering and $\Sigma_{v_2}$ is left-veering, secondly, $\Sigma_{v_2}$ is right-veering and $\Sigma_{v_1}$ is left-veering.

We conclude the statement for the tree $A$ by applying \cref{crit} and induction on the number of vertices of $A$. For $n=1$ $A=\{v\}$ consists of a single right or left veering $\Sigma_v\subset M_v$), and, as $\Sigma_A=\Sigma_v$, a fixed isotopy class of arcs for $\Sigma_A$ is one for $\Sigma_v$.

Say $A$ has $n$ vertices and the statement holds for all such $A$ with $n-1$ or less vertices. Pick a leaf $v$ of $A$ and consider the tree $A'$ with $n-1$ vertices obtained from removing $v$ (and the edge $e$ to it) from $A$.
Consider $\Sigma_{A'}\subset M_{A'}$ and denote by $\phi'$ its monodromy. Let $v'$ and $v$ be the vertices in $e$, say, without loss of generality, $e=(v',v)$ (as can e.g.~be assumed by switching all orientations on edges and all surfaces).
We let $f\colon P\to \Sigma_{v'},f'\colon P\to \Sigma_v$ be the summing region and note that by composing $f$ with the canonical inclusion $\Sigma_{v'}\subseteq\Sigma_{A'}\subset M_{A'}$, that we have a summing region that defines a Murasugi sum between $\Sigma_{A'}$ and $\Sigma_v$ that is (up to canonical isotopy) $\Sigma_A\subset M_A$.
By the setup, $\Sigma_v$ and $\Sigma_{v'}$ veer oppositely, say, without loss of generality, $\Sigma_v$ is right veering and $\Sigma_{v'}$ is left veering.
By \Cref{lemma:diskcomposition} (respectively~\eqref{eq:phi=phi1phi2}), $\phi'$ equals some composition of the $n-1$ $\Phi_{\widetilde{v}}$---the partial monodromies extended to $\Sigma_{A'}$---where $\widetilde{v}$ are the vertices of $A'$. As the plumbing regions are disjoint, this means $\phi$ is left-veering in $P$ (since $\Phi_{v'}$ is left-veering in $P$).
Thus, if $a\in \mathcal{A}(\Sigma_A)$ is fixed, the assumptions of \cref{crit} are satisfied ($\phi_v$ is right veering, hence in particular it is right-veering in $P$, and $\phi^{-1}$ is right-veering in $P$), and $a\cap \Sigma_{A'}$ and $a\cap \Sigma_{v}$ are collections of fixed isotopy classes of arcs in $\mathcal{A}(\Sigma_A')$ and $\mathcal{A}(\Sigma_v)$, respectively.
By the induction hypothesis, the former give collections $a\cap\Sigma_{\widetilde{v}}$ for all vertices in $A'$ (i.e.~all vertices of $A$ different from $a$), hence, we have that $a \cap \Sigma_v$ is collection of fixed isotopy classes of arcs in $\mathcal{A}(\Sigma_v)$ for all vertices $v$ of $A$.

It remains to observe that if $a$ is a fixed isotopy class of arcs in $\Sigma_{T_\mathrm{max}}$ for $T_\mathrm{max}\in A$, that then, for all vertices of $T_\mathrm{max}$, $a\cap\Sigma_v$ is a collection of fixed isotopy classes of arcs. This is immediate from the fact that its monodromy $\phi$ is a composition of $\Phi_v$ partial maps on $\mathcal{A}\Sigma_{T_\mathrm{max}}$ all of which are right-veering, hence as $a$ is fixed by the composition, it is fixed by each individual $\Phi_v$. The later is equivalent to $a\cap\Sigma_v$ is a collection of fixed isotopy classes of arcs for $\Sigma_v$.

For the `in particular', note that any (isotopy class of an) arc that essentially intersects any of the plumbing regions has the property that it intersects at least one surface in (a collection of isotopy classes of an) arc that is not boundary parallel. Hence, assuming towards a contradiction that there exists a fixed isotopy class of arcs $a\in\mathcal{A}$ that is essential, then $a\cap \Sigma_v$ contains at least one essential fixed isotopy class of arcs for at least one $v$, contradicting the assumption that $\Sigma_v$ has no such fixed isotopy class (as it is strictly right or left-veering). 
\end{proof}

\subsection{Arborescent links are prime}

As an application of \Cref{thm:treeofveeringsurfaces}, we discuss the primeness of arborescent links.
We do so without recalling the construction of arborescent links in detail, as we see this as an illustration of the technique rather than a main result of the text. In particular, since primeness for most arborescent links can be derived from most of them being hyperbolic, as discussed in \cite{Futer_2008}, we do not claim originality but rather a different perspective.
A general reference for the construction of arborescent links (certain links associated with a planar rooted tree with integer labeled vertices) we refer to~\cite[Ch.~12]{bs}.
\emph{Arborescent links} are the links $L_G$ in $S^3$ (well-defined up to ambient isotopy) associated to a plane tree $G$ with non-zero integer labeled vertices, i.e.~a choice of embedding of a finite tree with vertices labeled by elements in~$\Z\setminus \{0\}$. The link $L_G$ can be described by placing an unknotted $k$-twisted band in the plane at each vertex (with $k$ being the label of the corresponding vertex) and essentially plumbing them together whenever they share an edge; compare with~\cite{MR2165205} where this construction is explained, and consult~\cite[Ch.~12]{bs} for a detailed account. Note that the order of plumbing is indicated by the root: one starts with the root vertex and plumbs the other bands iteratively to it.

In the language of Murasugi sums, arborescent links are those arising from essentially plumbing together oriented unknotted annuli (understood as incompressible surfaces in $S^3$) guided by the planar rooted tree as follows.
It is not hard to see that every such $G$ in fact determines a tree $T$ as in the last section, where we chose orientations on the edges pointing away from the root. The corresponding surface $\Sigma_G\subset S^3$ is an incompressible surface (as an iterative Murasugi sum of incompressible surfaces, which all twisted annuli are) with boundary $L_G$. More details in the next paragraph.

The vertices of $T$ are chosen to be open books $(A,S^3)$, where $A$ is the unknotted annulus with $k$ full twists. The edges are chosen according to the edges of $G$. The choice of labeling of the edges is where all the subtlety lies. The fact is that the plane tree $G$ induces a circular order on the edges at a vertex of $G$ (and hence of $T$). We choose the plumbing regions $P_1, \ldots, P_n$ in an annulus $A$ that is the surface of a vertex $v=(A,S^3)$ of $T$ the  we take as the labels of the edges $e_1$ to $e_n$ that contain $v$ to be $P_1, \ldots, P_n$ in any order such that 
the circular order induced by the orientation of $A$ equals the circular order induced by $G$. This process yields an open book well-defined up to equivalence, as the following argument shows. Any other choice of order of the $P_i$ that also respects the circular order described above, amounts to changing the plumbing regions by applying an isotopy to $A$, hence the resulting surface in $S^3$ is the same (up to a canonical isotopy induced by the isotopy of $A$).

    With this setup, the fact that arborescent links are prime is immediate from \Cref{thm:treeofveeringsurfaces}.

\begin{proposition}
\label{prop:arblinksareprime}
All arborescent links are prime.
\end{proposition}
\begin{proof}Let $L$ be an arborescent link. It is the boundary of the incompressible surface $\Sigma_T\subset M_T=S^3$, where $T$ is as described above. Each of the vertices $v$ consists of an annuls that is strictly right-veering and strictly left veering (only boundary parallel arcs go to the other side. By \Cref{thm:treeofveeringsurfaces}, $\Sigma_T$ is irreducible which (as $M_T=S^3$ is prime) is equivalent to $L$ being a prime link.
\end{proof}

In fact, the same argument works for any planar rooted tree of annuli where none of the annuli are zero-framed and unknotted. Indeed, any such annulus is incompressible, and only boundary parallel arcs go to the other side, and the above proof applies verbatim.

We note that this section is analogous to the corresponding section in~\cite{FLOR}, with the crucial difference that we do \emph{not} assume fiberedness as we no longer need open books in our setup. The reader might notice that in the fibered case, one considers planar trees rather than rooted planar trees. This is due to the fact that the order of plumbing does not change the isotopy type of the resulting surface in the fibered case; this is not so in general, as e.g.~essentially plumbing together a twice twisted band and a thrice twisted band show (yields non-isotopic oriented (!) surfaces) or essentially plumbing together three bands that are twice, thrice, and five-times twisted, respectively.

\section{Strongly quasipositive surfaces are right-veering} \label{sec:sqp}
As a main technical criterion to detect irreducibility of a surface, we discuss the irreducibility of Murasugi sums of right-veering and left-veering surfaces in terms of irreducibility of the summands. 

For our application to Cromwell's conjecture, we in particular need to understand large classes of surfaces (e.g. coming from diagrams) in $S^3$ that are right-veering. One such class are the so-called strongly quasipositive surfaces. Indeed, Seifert surfaces associated with positive link diagrams are strongly quasipositive~\cite{Rudolph}, and the main result of this paper addressing Cromwell's conjecture concerns link diagrams that are built (by Murasugi sums) out of positive and negative link diagrams.

When they are fibered, strongly quasipositive surfaces are known to be right-veering. We extend this result to all strongly quasipositive surfaces in the following proposition. We note that our proof is completely topological, and makes no use of contact topology. Hence, even for the case of fibered surfaces, this is, as far as the authors know, a new strategy to establish right-veeringness. 

\begin{proposition}\label{prop:sqp=>rv}
If $\Sigma$ is a strongly quasipositive surface, then $\Sigma$ is  right-veering, i.e.~if an arc $a$ in $\Sigma$ goes to the other side, then~$a\leq \phi(a)$.
Furthermore, strongly quasipositive surfaces are %
strictly right-veering if and only if they are irreducible.
\end{proposition}
We provide context for strong quasipositivity, before proving the above result.
Introduced by Rudolph~\cite{zbMATH03826789}, strongly quasipositive surfaces have become interwoven with a large part of low-dimensional topology, due to their connections to complex plane curves~\cite{zbMATH03796828,zbMATH01643188} and contact geometry~\cite{Hedden}.

While originally defined as the surfaces that are naturally associated with a particular braid representation of links~\cite{zbMATH03826789}, we use the following equivalent characterization: a surface $\Sigma$ in $S^3$ is \emph{strongly quasipositive} if, up to isotopy, it arises as a full subsurface of the fiber surface of a positive $(n,n)$-torus link~\cite[Characterization Theorem]{Rudolph_CharacterizationofQPSSConstructions3}. Here, a subsurface $F\subseteq \Sigma$ is said to be \emph{full}\footnote{The notion of full subsurfaces is equivalent to $F$ being a $\pi_1$-injective subsurface. We follow Rudolph and use the term full.}, if every simple closed curve in $F$ that bounds a disk in $\Sigma$ also bounds a disk in~$F$.
Even for annuli (i.e.~$\Sigma\cong S^1\times[0,1]$), the notion of quasipositivity is interesting and related to contact geometry properties of the knot given as their core. In the argument below, we will need the following.

\begin{Example}({\cite{zbMATH00165700}})\label{Ex:sqpannuli}
Let $A\subset S^3$ be an annulus in $S^3$ that is unknotted, i.e.~its core is unknotted. Then $A$ is strongly quasipositive if and only if its framing is strictly negative. Here, the framing $k\in \mathbb{Z}$ of $A$ is defined as minus the linking number of the two boundary components of~$A$.
\end{Example}
    
With these preliminaries on strongly quasipositive surfaces out of the way, we phrase a right-veeringness result for families of surfaces in a fixed $3$-manifold that in particular implies \Cref{prop:sqp=>rv}.

\begin{theorem}\label{thm:rvnesscrit}Let $M$ be a $3$-manifold and let $\mathcal{P}$ be a family of (isotopy classes of) incompressible compact surfaces in $M$ that is closed under positive stabilization. All surfaces in $\mathcal{P}$ are right-veering, if there is no surface in $\mathcal{P}$ that contains an essential unknotted 0-framed annulus.

\end{theorem}

\begin{remark} \label{rmk:soberingarc}
    In fact, as the proof will reveal, we show something slightly stronger about the unknotted $0$-framed annulus that can be phrased in the language of monodromies. Let $\mathcal{P}$ a family as above. Then all surfaces in $\mathcal{P}$ are right-veering, if no element $\Sigma\in \mathcal{P}$ has the property \eqref{eq:P} described below (see also \Cref{fig:sobering_arc}).   \begin{equation}\label{eq:P}\parbox{11.5cm}{
        There exists $a \in \mathcal{A}(\Sigma)$ such that $a$ goes to the other side (i.e.~$\phi(a)$ is defined), $a$ is not fixed (i.e.~$\phi(a)\neq a)$), one endpoint of $a$ is right-veering and the other is left-veering (i.e.~$\phi(a)>a\,\&\,\phi(\overline{a})<\overline{a}$ or~$\phi(a)<a\,\&\,\phi(\overline{a})>\overline{a}$), and $a$ and $\phi(a)$ have representatives $\alpha$ and $\alpha'$ that do not intersect in their interior.
}\end{equation}
\begin{figure}[ht]
    \centering
    \includegraphics[width=0.2\linewidth]{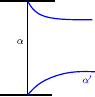}
    \caption{Arcs $\alpha$ (black) and $\alpha'$ (blue) that represent an isotopy class $a=[\alpha]$ and its image $\phi(a)=[\alpha']$, respectively, for an $a$ that satisfies \eqref{eq:P}.}
\label{fig:sobering_arc}
    \end{figure}

    This is, of course, an `if and only if'.
In other words, for a family $\mathcal{P}$ as above, containing at least one surface that is not right-veering is equivalent to containing a surface that satisfies~\eqref{eq:P}, which is of course not right-veering.
    
\end{remark}
An example of such a family are the strongly quasipositive surfaces in~$S^3$; below, we deduce \Cref{prop:sqp=>rv} from \Cref{thm:rvnesscrit}.
While \Cref{prop:sqp=>rv} is the only application of \Cref{thm:rvnesscrit} used in other sections of this paper, we consider \Cref{thm:rvnesscrit} (and the `if and only if' variant provided in~\cref{rmk:soberingarc} and the introduction) of independent interest, as
considering incompressible surfaces up to positive stabilization is rather widespread,
often adding in closedness under the inverse operation, so-called destabilization.
Most celebratedly, for a closed $3$-manifold~$M$, by Giroux correspondence open books of $M$ (i.e.~compact fiber surfaces) up to positive stabilization and destabilization are in 1-to-1 correspondence with isotopy classes of contact structures on~$M$~\cite{zbMATH01789982};
concerning right-veeringness, the families of open books (closed under (de-)stabilizations) that are right-veering correspond exactly to so-called tight contact structures by a celebrated result of Honda, Kazez, and Mati\'{c}~\cite{HKM}.
At the end of this section, we explain how~\cref{thm:rvnesscrit}
does in fact provide an avenue to prove this latter result (and its generalization to partial open books);
see \cref{cor:rvtight}.
Beyond open books, there is further motivation (e.g.~from Heegaard Floer theory) to  study surfaces (that are, in general, not pages of open books) up to positive stabilization;
compare e.g.~Baldwin's recent blog post~\cite{baldwinblog}.

We now return to the proof of the application of \cref{thm:rvnesscrit} needed in \cref{sec:alternative}.
\begin{proof}[Proof of \Cref{prop:sqp=>rv}]
Recall that the family of isotopy classes of strongly quasipositive surfaces in $S^3$ are closed under taking full subsurfaces (this is e.g.~evident from the characterization as the full subsurfaces of positive torus link fibers) and are closed under taking positive stabilizations (since if $\Sigma$ is strongly quasipositive, then so is~$\Sigma\#_PH$, where $H$ is any strongly quasipositive surface, in particular in case $H$ is a positive Hopf band and $P$ is a square---the definition of positive stabilization). Moreover, the unknotted 0-framed annulus in $S^3$ is not strongly quasipositive, compare \Cref{Ex:sqpannuli}. Hence, by \cref{thm:rvnesscrit}, all elements of $\mathcal{P}$ are right-veering.

The `furthermore' part now follows quickly: for a right-veering surface, strict right-veeringness is equivalent to the absence of intervals $a$ such that~$\phi(a)=a$, which in turn is equivalent to irreducibility.
\end{proof}

 \begin{proof}[Proof of \Cref{thm:rvnesscrit}]

    Let $\Sigma$ be a compact incompressible surface in a $3$-manifold~$M$, and suppose that there exists $a\in\A(\Sigma)$ such that~$\phi(\overline{a}) < \overline{a}$. We will perform a series of positive Hopf plumbings until we get an arc that exhibits an essential $0$-framed unknotted annulus. Writing $\widetilde{\Sigma}$ ($\supset \Sigma$) for the surface resulting from the positive Hopf plumbing and $\widetilde{\phi}$ for its monodromy, we note that by \Cref{eq:phi=phi1phi2}
    the new image $\widetilde{\phi}(a)$ of $a$ after each step will be the old one postcomposed with a positive Dehn twist along a curve that is the plumbing arc union the core of the $1$-handle added in the plumbing.
     
     Up to performing a boundary parallel positive Hopf plumbing on one of the endpoints, we can assume that $a$ is right-veering at the starting point, say~$p$, and $\overline{a}$ is left-veering at its starting point, which we denote by~$q$. We look at the sign of the intersections of $\phi(a)$ and $a$ other than $p$ and $q$. Here, choose $\phi(a)$ to have the orientation induced from~$a$, and in the rest of the proof intersections are always understood to be interior intersections, that is neither $p$ nor $q$. We choose representatives, so that we may think of the intersections as explicit points.

     Let $\alpha$ and $\alpha'$ be minimally transversely intersecting representatives of $a$ and~$\phi(a)$.
     Our first aim is to perform Hopf plumbings that inductively remove all negative intersection points, without changing the fact that at~$q$, $\overline{a}$ veers to the left. In other words, each Hopf plumbing we perform is such that an arc $\widetilde{\alpha}'\subset\widetilde{\Sigma}$ representing $\widetilde{\phi}(a)$ (i.e.~the image of $\alpha'$ under the relevant Dehn twist on $\widetilde{\Sigma}$), after it is isotoped to intersect $\alpha\subset \Sigma\subset\widetilde{\Sigma}$ minimally, lies to the correct side of $\alpha$ at $q$: the tangent vector of $\widetilde{\alpha}'$ is to the left of that of $\alpha$.
     Afterwards, when there are only positive intersections left, our second aim is to perform further positive Hopf plumbings that remove all remaining intersections.

     \subsection*{Negative intersections}
     We first assume that there exists at least one negative intersection point~$x \in \alpha \cap \alpha'$. If there are multiple, we choose $x$ as follows.
     In case, the first intersection between $\alpha$ and $\alpha'$ that occurs along $\alpha$ (starting from $p$) is negative, choose this to be $x$. If not, we choose an $x$ such that the intersection point that occurs before it on $\alpha$, which we denote by $y$, is positive. We will later refine this choice when needed.

     We consider the connected component $R$ of $\Sigma\setminus(\alpha\cup\alpha')$ that is to the left of $\alpha$ and $\alpha'$ at~$x$ and discuss two cases depending on whether $R$ contains part of the boundary of $\Sigma$.

     \subsubsection*{$R$ contains part of the boundary of~$\Sigma$}
     
     We consider the arc $\beta$ consisting of two subarcs: the first starts slightly to the right of~$p$, then follows $\alpha'$ until $x$ and the second one is any choice of arc in $R\cup \{x\}$ that starts at $x$ and ends on~$R\cap \partial \Sigma$.
     Using $\beta$ as an arc for positive Hopf plumbing, in the resulting surface $\widetilde{\Sigma}$ with monodromy $\widetilde{\phi}$ we have $\widetilde{\phi}(\overline{a})<\overline{a}$ and the number of negative intersections between $a=[\alpha]\in\A\left(\widetilde{\Sigma}\right)$ and $\widetilde{\phi}(a)$ is less; see \Cref{fig:easy_removal}. In fact, the plumbing removes the intersection corresponding to $x$ and all that are before it on~$\alpha'$
\begin{figure}[ht]
    \centering
    \includegraphics[width=0.5\linewidth]{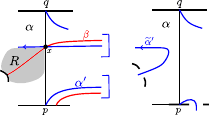}
    \caption{Removing $x$ if $R$ contains part of~$\partial \Sigma$. Right: the arc $\beta$ (red). Right: $\widetilde{\alpha'}$ (blue) is a representative of $\widetilde{\phi}(a)$ with minimal intersection with~$\alpha$. Note that any other intersection points that were between $p$ and $x$ running along $\alpha'$ are have no corresponding intersections between $\alpha$ and $\widetilde{\alpha'}$. We use the convention that square brackets with the same color are identified.}
\label{fig:easy_removal}
    \end{figure}
     
     \subsubsection*{$R$ does not contain part of the boundary of~$\Sigma$}
     Note that then $x$ cannot be the first intersection point (as for such a negative intersection point $x$, $R$ does contain part of the boundary of~$\Sigma$).
     Recall that due to our choice of $x$, the intersection point $y$ that occurs before $x$ (along $\alpha$), is positive by hypothesis.

     In this case we build an arc $\beta$ as follows. Start slightly to the right of $p$ and follow $\alpha'$ until one reaches the intersection point, out of $x$ and $y$, that comes second. Then follow $\alpha$ until the other one is reached, from where in reverse orientation one follows $\alpha'$ back to $p$ and ends slightly to the left of $p$ (on the boundary~$\partial \Sigma$). This involves a unique intersection point between $\beta$ and $\alpha'$, which we choose to be at $x$.

     Plumbing a positive Hopf band along $\beta$ yields a surface with a monodromy $\widetilde{\phi}$ such that $a$ and $\widetilde{\phi}(a)$ have at least one less negative intersection than $a$ and~${\phi}(a)$. Indeed, the intersections corresponding to $x$ and $y$ and all intersections lying between them on $\alpha'$ are removed while one positive intersection is added (stemming from where the last part of $\beta$ hits $\alpha$).
     For visualization, it helps to distinguish cases based on in which order $x$ and $y$ occur on~$\alpha'$;
     see \Cref{fig:removing_negative1,fig:removing_negative2}, respectively. Positive Hopf plumbing along $\beta$ reduces the number of intersections.
     If this positive Hopf plumbing is such that $\overline{a}$ stays left-veering, then we are done with this case.
     
     However, it could be that it becomes right-veering, and we need to pick a slightly different $\beta$, which we discuss in the remainder of the case dealing with negative intersections.
\begin{figure}[b]
    \centering
    \includegraphics[width=0.6\linewidth]{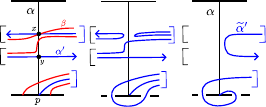}
    \caption{Removing $x$ if $x$ comes before $y$ on~$\alpha'$.}
\label{fig:removing_negative1}
    \end{figure}
\begin{figure}[b]
    \centering
    \includegraphics[width=0.6\linewidth]{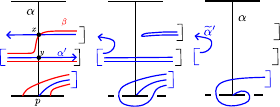}
    \caption{Removing $x$ if $y$ comes before $x$ on~$\alpha'$.}
\label{fig:removing_negative2}
    \end{figure}
In this case, we must have a rectangle with corners $x,y,q$, and another point $w \in \alpha \cap \alpha'$ that comes before $x$ and $y$ in $\alpha'$; see \Cref{fig:BigonCase}. (Since this is exactly what, in the surface obtained by plumbing a positive Hopf band along the above proposed $\beta$, corresponds to the existence of a bigon %
between the new $\widetilde{\alpha}'$ and $\alpha$ that ends at $q$, which is needed for $\overline{a}$ to stop being left-veering).

    \begin{figure}[ht]
    \centering
    \includegraphics[width=0.4\linewidth]{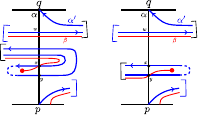}
    \caption{The cases where $\overline{a}$ becomes right-veering if we perform the previous plumbings. The rectangle whose vertices are $x,y,q,w$ could contain more intersection points.}
\label{fig:BigonCase}
    \end{figure}

Here is where we refine our choice of $x$ and $y$: if there is any negative intersection point $x'$ that occurs (on $\alpha$) after a positive intersection point $y'$ that does not feature a rectangle as described above (in other words, constructing $\beta$ using this pair of points rather than an arbitrary $x$ and $y$ leads to a plumbing that preserves left-veeringness at $q$), we choose those as our $x$ and $y$ and argue as above. Hence, we are left with the case that all pairs $x'$ and $y'$ of this type (that is $x'$ is negative, $y'$ is positive and $x'$ occurs just after $y'$ on $\alpha$) features a rectangle as described above. In fact, this means that the corresponding rectangles are nested, and we pick $x$ and $y$ to be the pair that features the rectangle that contains all others; equivalently, on $\alpha'$ the pair $x$ and $y$ occurs between all other such pairs $x'$ and $y'$; see \cref{fig:big_rectangle}.
\begin{figure}[ht]
    \centering
    \includegraphics[width=0.3\linewidth]{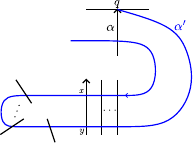}
    \caption{The points $x$ and $y$ are chosen to be the ones that feature the largest of the nested rectangles. The case where $x$ occurs before $y$ on $\alpha'$ is depicted, for the other case, exchange the labels of the two points and the orientation of the segment of $\alpha$ delimited by $x$ and $y$.}
\label{fig:big_rectangle}
    \end{figure}

We next describe an arc $\beta$, and the existence of this (largest) rectangle is exactly what will allow us to conclude that after the plumbing $\overline{a}$ is still left-veering at $q$. For this, let $\alpha'_{x,y}$ be the subsegment of $\alpha'$ delimited by $x$ and $y$.

 We take $\beta$ as before to go along $\alpha'$ until either $x$ or $y$, which ever appears first on~$\alpha'$. We consider $S_1,S_2,\dots, S_n$ to be the $n$ connected components of $\Sigma\setminus (\alpha\cup \alpha')$ that arise to the left of $\alpha'_{x,y}$ when going along $\alpha'_{x,y}$, where $n$ is a positive integer.
 We choose one of these components, and call it $S$, as follows:
 if there exist components $S_i$ for which the part of the boundary that consists of segments of $\alpha'$ is not contained in $\alpha'_{x,y}$, pick $S$ to be the first such component. If such a component does not exist, it turns out that $x$ and $y$ are consecutive intersections not only on $\alpha$ but also on $\alpha'$; in other words $\alpha'_{x,y}$ only intersects $\alpha$ in its end points $x$ and $y$. In particular, there is only one component $S_i$ as above (that is, $n=1$) and we take $S$ to be that component. Before we go on, we argue that indeed such an $S$ exists. This is the content of the next three paragraphs.
 
Assume that there is at least one intersection $z_0\in \alpha\cap\alpha'$ that is in the interior of $\alpha'_{x,y}$, since otherwise, we are in the latter case and we take $S$ to be the unique connected component. We consider the following (unoriented) segment $\alpha_{z_0}$ in $\alpha$: it starts at $z_0$ then follows $\alpha$ to the left of $\alpha'_{x,y}$ and ends at the first intersection with $\alpha'$ that occurs, which we denote by $z_1$. If $z_1$ is not in the interior $\alpha'_{x,y}$, then we find a connected component $S_i$ that in its boundary contains parts of $\alpha'$ that occur before or after $x$ and $y$, as desired; see left-hand side of \cref{fig:z0z1}.
\begin{figure}[ht]
    \centering
    \includegraphics[width=0.5\linewidth]{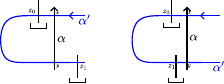}
    \caption{Left: the segment $\alpha_{z_0}$ when $z_1$ is not in the interior of $\alpha'_{x,y}$. Right: the segment $\alpha_{z_0}$ when $z_1$ is in the interior of $\alpha'_{x,y}$.
    As in \cref{fig:big_rectangle}, only the case where $x$ occurs before $y$ on $\alpha'$ is depicted.}
\label{fig:z0z1}
    \end{figure}
Indeed, at least one of the components $S_i$ (or the component, in case they coincide) touching $z_0$ to the left of $\alpha'_{x,y}$ is as desired.

If $z_1$ is in the interior $\alpha'_{x,y}$, note that the sign of $z_1$ has to be the same as that of $z_0$; see right-hand side of \cref{fig:z0z1}. Indeed, if the signs were opposite, we would have found a pair $x'$ and $y'$ (namely, $x'=z_0$ and $y'=z_1$ if the sign of $z_0$ is negative and $x'=z_1$ and $y'=z_0$, otherwise) that occurs on $\alpha'_{x,y}$, which cannot exist by our choice of $x$ and $y$. We consider the following (unoriented) segment $\alpha_{z_1}$ in $\alpha$: it starts at $z_1$ and follows $\alpha$ to the left of $\alpha'_{x,y}$ (in other words, since the sign of $z_0$ and $z_1$ agree, away from $\alpha_{z_0}$) and ends at the first intersection with $\alpha'$, which we denote by $z_2$.

Iteratively, we construct further such segments $\alpha_{z_j}$ as long as $z_j$ lies in the interior of $\alpha'_{x,y}$. At some point we run out of intersections in the interior of $\alpha'_{x,y}$; hence, we will have found a segment $\alpha_{z_j}$ that at $z_j$ starts to the left of  $\alpha'_{x,y}$ but ends at a point $z_{j+1}$ not in the interior of  $\alpha'_{x,y}$, in which case we find a connected component $S_i$ as desired. This concludes the argument that $S$ exists as desired.
 
 We take $\beta$ to further follow $\alpha'_{x,y}$ until $S$ is to the left and then we take it to enter~$S$.
If $S$ has a part of the boundary contained in $\alpha'\setminus\alpha'_{x,y}$, and that part comes before $x$ and $y$ (that is, the region $S$ has a side part of $\alpha'$ that comes before $x$ and $y$), then we take $\beta$ to follow $\alpha'$ back after exiting $S$ through such a part (this is not drawn, for an example see \Cref{fig:Problem_example}). This reduces the number of intersections (and in particular removes $x$). Indeed, all of the intersections before $x$ and $y$ get removed. Then, some of them get added again, but since $\beta$ comes out somewhere that comes before $x$ and $y$, there are fewer of them, and moreover they appear with the same sign as before (potentially we get an extra intersection added coming from an extra intersection point $\beta \cap \alpha$ near $p$, as it happened also in \Cref{fig:removing_negative1,fig:removing_negative2}, but crucially this point will always be positive). Moreover, this plumbing keeps $\overline{a}$ left-veering.
 
If $S$ has a part of the boundary contained in $\alpha'\setminus\alpha'_{x,y}$, but all of those come after $x$ and $y$, we choose $\beta$ to follow $\alpha'$ forward after exiting $S$ through such a part until we end slightly to the left of $q$, see \Cref{fig:Problem_second_example}. Note that we do indeed end to the left of $q$ because to the right we have the rectangle that did not allow us to do our original strategy. Then, all intersections before $x$ and $y$ get removed. Now, as in the previous case, some of them get added again, but not all of them (and not new ones). Moreover, they appear with the same sign as before.

Finally, if all parts of the boundary of $S$ contained in $\alpha'$ are contained in $\alpha'_{x,y}$, then our $S$ is not a disc (since a disc would be a bigon as its boundary consist of two segments, one contained in $\alpha$ and the other, namely $\alpha'_{x,y}$, is contained in $\alpha'$). It is a once punctured genus $g\geq 1$ surface, and we indicate the existence of topology by a red circle, which $\beta$ enters; see \cref{fig:BigonCase} for an example. %

In this case, we can do two plumbings to keep the same number of intersections but now the region is of one of the previous two types; see \Cref{fig:In_between}. %

    \begin{figure}[ht]
    \centering
    \includegraphics[width=0.5\linewidth]{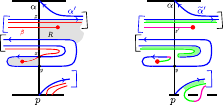}
    \caption{On the left, the arc $\beta$, which, after intersecting $x$, comes out in a region with a vertex $z$ that comes before $x$ and $y$ along $\alpha'$. On the right, the green bigon exhibits that all intersection points up to and including $x$ are removed. Then, the purple segment adds every intersection point up to and including $z$ (labeled $z'$ in $\widetilde{\alpha'}$) and by construction all of these have the same signs as the original ones. Nothing changes after this, and in particular $\overline{a}$ is still left-veering.}
\label{fig:Problem_example}
    \end{figure}

    \begin{figure}[ht]
    \centering
    \includegraphics[width=0.5\linewidth]{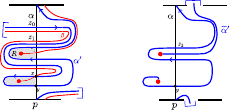}
    \caption{Left, the plumbing arc $\beta$. Right, all the intersection points before $x$ and $y$ get removed, and only the ones between $z_0$ (excluded) and $z_1$ (included and labeled $z'_1$ in $\widetilde{\alpha'}$) get added back (in the figure only $z_1'$ is shown, but there could be more), and with the same sign. Moreover, $\overline{a}$ stays left-veering.}
\label{fig:Problem_second_example}
    \end{figure}

    \begin{figure}[ht]
    \centering
    \includegraphics[width=0.7\linewidth]{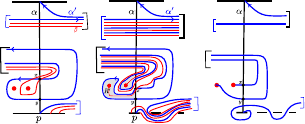}
    \caption{Doing two plumbings when all sides of $S$ are between $x$ and $y$ to reduce to one of the previous cases (at the expense of an extra positive intersection point near $p$). Note that the second plumbing is also one of the previous cases, as the region $R'$ has a side that comes between $x'$ and $y$.}
\label{fig:In_between}
    \end{figure}

     \subsection*{Positive intersections}
     By the above construction, after sufficiently many positive Hopf plumbings, for the rest of the proof we may and do assume that all intersections between $a$ and $\phi(a)$ are positive. We will inductively reduce the number of those intersections until we have none left, without changing the fact that at~$q$, $\overline{a}$ veers to the left.

Suppose all intersections are positive. Let $y_1$ be the closest point to $q$ (along $\alpha$). Then, the region $R_1$ which is to the left of both $\alpha$ and $\alpha'$ at $y_1$ contains a part of the boundary. Take a plumbing along an arc $\beta$ as follows. Start slightly to the right of $p$, follow $\alpha'$ until right before $y_1$, then enter $R_1$ by intersecting $\alpha'$ and go to the boundary. If $y_1$ is not the first intersection point running along $\alpha'$, this removes all previous intersection points. If it is, the intersections are unchanged but now if we consider the previous point $y_2$ (along $\alpha$), the associated region $R_2$ (after the plumbing) which is to the left of $\alpha$ and $\widetilde{\alpha}'$ at $y_2$
contains a part of the boundary, and $y_2$ is not the first point we encounter running along $\alpha'$, so we can repeat. See \Cref{fig:removing_positive} for an illustration. Note that since $y_1$ is not removed, in particular $\alpha$ stays left-veering at $q$.

    \begin{figure}[ht]
    \centering
    \includegraphics[width=0.6\linewidth]{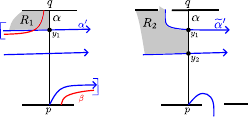}
    \caption{The plumbing along $\beta$ removes all intersection points that come before $y_1$ along $\alpha'$ and moreover ensures that the region $R_2$ now contains part of the boundary.}
\label{fig:removing_positive}
    \end{figure}

    Thus the only case to that remains to be discussed is when there is a unique intersection point $y$. We consider the region $R$ that lies to the left of $\alpha'$ and to the right of $\alpha$ at $y$. If $R$ does not have only one puncture, then $R = \Sigma \setminus (\alpha \cup \alpha')$, in particular it contains the boundary and we are done, as the plumbing strategy when the region contains part of the boundary (see \Cref{fig:easy_removal} for an illustration) works in the same way for a positive intersection.
    If it does, then $R$ is a genus $g\geq 1$ surface with one puncture. Then we take $\beta$ to be an arc starting slightly to the right of~$p$, following $\alpha'$ up to~$y$, and then following an arc in $R\cup \{y\}$ that starts at $y$ and ends slightly to the right of~$q$, such that $\beta \cap \Sigma_{g,1}$ is an essential non-separating arc, see \Cref{fig:subsurface}. The fact that $\beta \cap \Sigma_{g,1}$ is essential ensures that $\overline{a}$ is still left-veering for the new monodromy~$\widetilde{\phi}$, and the fact that it is nonseparating ensures that the new region $R$ contains the boundary, so we have reduced to the previous case and we are done.
    
            \begin{figure}[ht]
    \centering
    \includegraphics[width=0.6\linewidth]{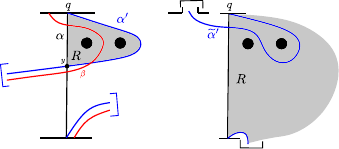}
    \caption{The case where there is a unique (positive) intersection point, and $R$ has one puncture and does not contain parts of the boundary of~$\partial \Sigma$. The two black circles represent one of the genera in~$R$, so that $\beta \cap R$ is non-separating. On the top right, we perform the positive Hopf plumbing, which can be seen on the bottom to remove all intersections except the last one, and due to our choice of plumbing, the new $R$ is such that $\mathrm{closure}(R)$ has only one puncture and in fact it contains parts of~$\partial \Sigma$.
    }
\label{fig:subsurface}
    \end{figure}

    \subsection*{No intersections}
    Therefore, by performing Hopf plumbings, we obtain an arc $a$ that does not intersect $\phi(a)$ apart from their common end points; in other words, we found $a$ that satisfies \eqref{eq:P}. This in particular gives us a $0$-framed unknotted annulus. Indeed, take $\gamma, \gamma'$ representatives of $a, \phi(a)$ respectively, as in \eqref{eq:P}. The desired annulus $A$ is a neighborhood of~$\gamma \cup \gamma'$. It is clearly unknotted and its framing is equal to the framing of $A_D$, as we discuss in detail in the next paragraph.

    Let $D$ be the product disk $D$ bringing $\gamma$ to~$\gamma'$. Consider the annulus $A_D$ given by restricting the normal bundle of $D$ to a simple closed curve; see \Cref{fig:A_D}.
\begin{figure}[ht]
    \centering
    \includegraphics[width=0.3\linewidth]{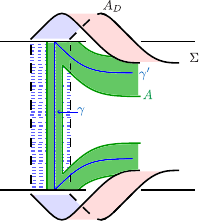}
    \caption{The surface $\Sigma$ and the annuli $A$ and~$A_D$, for minimally intersecting representatives $\gamma$ and $ \gamma'$ of $a$ and~$ \phi(a)$.}
\label{fig:A_D}
    \end{figure}
    Its core is isotopic to the core of~$A$, and their framing agree, by virtue of $A$ and $A_D$ being parallel away from a neighborhood of the start and endpoint of~$a$, up to a correction of $\pm\frac{1}{2}$ at the start and endpoint.
    Since  ${a}<\phi({a})$ and $\bar{a}>\phi(\bar{a})$ (or vice versa) the difference in framing at the start and endpoint is opposite and hence $A_D$ and $A$ have the same framing; see again \Cref{fig:A_D}.
\end{proof}

As hinted at above, the proof of \Cref{thm:rvnesscrit} gives some further contact geometric insight. The arc $a$ with property \eqref{eq:P} we obtain in the proof is a special case of a \emph{sobering arc} in the sense of Goodman~\cite{Goodman}. In particular, the existence of such an arc in a (possibly partial) open book shows that the contact structure that it supports is overtwisted.\footnote{Goodman constructs from his sobering arcs surfaces that violate Bennequin's inequality for tight contact structures. Due to the lack of interior intersections in the arc we construct, the resulting surface is already a disk, i.e.~an overtwisted disk.} Thus the proof of \Cref{thm:rvnesscrit} also recovers the following characterization of tight contact structures, originally due to Honda, Kazez, and Mati\'{c} \cite{HKM, HKMPartial}. We note that our proof bypasses the need for bypasses and Giroux correspondence; indeed, it is completely topological up until the point where we find the sobering arc and apply Goodman's construction to obtain the overtwisted disk. 

\begin{corollary} [\cite{HKM, HKMPartial}] \label{cor:rvtight}
    
    Let $(M,\xi)$ be a contact manifold, either closed or compact with convex boundary. Then $(M,\xi)$ is tight if and only if all its supporting (possibly partial) open book decompositions are right-veering. 
\end{corollary}

\begin{proof}
    Suppose $(M,\xi)$ is supported by a (partial) open book that contains an arc that is not right-veering. Then, our sequence of positive Hopf plumbings from \Cref{thm:rvnesscrit} gives a left-veering arc with no interior intersections (see \cref{rmk:soberingarc}). As discussed above such a so-called sobering arc provides an overtwisted disk in~$(M,\xi)$. We note that in the partial open book case, positive Hopf plumbings (together with our \Cref{lemma:diskcomposition}) coincide with the definition of positive stabilization given in \cite{HKMPartial}, which preserves the isotopy class of the contact structure.
    
    The other implication is easier and relies (in the closed case) on Eliashberg's classification of overtwisted contact structures \cite{Eliashberg}. Indeed it is done in the same way in both \cite{HKM} and \cite{Goodman}: if $(M,\xi)$ is overtwisted, it admits an open book which is obtained as a negative Hopf plumbing, and the cocore of the handle of this plumbing is clearly not right-veering.
    
    For the non-closed case, it suffices to show that a neighborhood of an overtwisted disk admits a partial open book that is not right-veering; see \cite{HKMPartial}.
\end{proof}

\begin{remark}\label{rem:wand}
    A proof of \Cref{cor:rvtight} with a similar construction has been done in the open book case by Wand \cite[Observation 4.6]{Wand}. We find it prudent to still include the proof above, as Wand's argument uses a different setup than \cref{thm:rvnesscrit} and came from a rather different motivation and perspective than our \cref{thm:rvnesscrit}.
\end{remark}

\section{Alternative links are visually prime}\label{sec:alternative}
This section is devoted to the proof of our main theorem.
Let us start by a brief review of Seifert's algorithm and alternative links.
Seifert's well-known algorithm \cite{Seifert_35} associates a Seifert surface $\Sigma_D$ with a link diagram $D$ as follows: one starts by resolving the crossings of $D$ in the oriented way, producing the so-called \emph{Seifert circles} in $\mathbb{R}^2\times\{0\}\subset\mathbb{R}^3$. Then one picks disjoint so-called \emph{Seifert disks} lying properly in 
$\mathbb{R}^2 \times [0,\infty) \subset \mathbb{R}^3$ whose boundaries are the Seifert circles.
Now $\Sigma_D$ is constructed as the union of those disks, and one twisted band (with core in $\mathbb{R}^2\times\{0\}$) for each crossing of~$D$.

We call a Seifert circle $C$ \emph{non-separating} (alternatively called type I) if all the crossings adjacent to $C$ are on the same side of~$C$.
Now let $C$ be a separating Seifert circle (alternatively called type~II).
One can form a new link diagram $D_e$ from $D$ by deleting the inside of $C$ (and smoothing all crossings adjacent to $C$ on the inside). Similarly, one forms $D_i$ from $D$ by deleting the outside of~$C$.
Consider the sphere $S$ formed by the Seifert disk $B$ of $C$ glued along $C$ to the reflection to $B$ across~$\mathbb{R}^2\times\{0\}$.
The complement $\mathbb{R}^3\setminus S$ has two components. One observes that $(\Sigma_D \cap \text{unbounded component}) \cup B$ equals $\Sigma_{D_e}$, and $(\Sigma_D \cap \text{bounded component}) \cup B$ equals~$\Sigma_{D_i}$, if the Seifert disks for $D_u$ and $D_b$ are picked appropriately.
In fact, one sees that 
$\Sigma_D$ is the Murasugi sum of $\Sigma_{D_e}$ and $\Sigma_{D_i}$ along a polygon~$P$ contained in~$B$ (this is a frequently used fact, see e.g.~\cite[Theorem 1]{zbMATH04122845}).
The number of sides of $P$ is given by how often one switches between crossings adjacent on the one side and on the other of $C$ when going once around~$C$.
Decomposing the surface $\Sigma_D$ in this way as Murasugi sum, simultaneously along all its separating Seifert circles, naturally leads to the following.
\begin{definition}\label{def:td}
Let $D$ be a \emph{non-split} link diagram, i.e.~one for which the immersed circles in the plane (the diagram without the crossing information) is connected. We define a tree $T(D)$ of surfaces
associated with $D$ (cf.~\cref{sec:trees}).
Denote by $\mathcal{C}\subset\mathbb{R}^2$ the union of all Seifert circles of~$D$.
Then, for every component $v$ of $\mathbb{R}^2\setminus \mathcal{C}$ containing crossings of~$D$,
let $T(D)$ have a vertex $(S^3, \Sigma_{D_v})$, where $D_v$ is the link diagram obtained from $D$ by deleting the other side of all boundary circles of~$v$.
For every separating Seifert circle $C$ separating components $v$ and $w$, let $T(D)$ have an edge between the vertices corresponding to $v$ and $w$. The edge label is the summing region contained in the Seifert disk of~$C$. The edge is oriented from $v$ to $w$ (or oppositely) if $v$ lies to the right (or left) of $C$ when facing in the direction of the orientation of $C$.
\end{definition}
By construction, we have that
for all non-split diagrams~$D$, the Seifert surfaces
$\Sigma_D$ and $\Sigma_{T(D)}$ are isotopic.

Now, a non-split link diagram $D$ is called \emph{alternative} if all of the $D_v$ are alternating; and a split diagram is alternative if all of its split components are.
Note that since $D_v$ is special (i.e.~all of its Seifert circles are non-separating) and non-split, $D_v$ is alternating if and only if it is positive or negative.
This definition of alternative is equivalent to the original definition~\cite[Def.~9.1]{zbMATH03854015}.
For alternative $D$, the $\Sigma_{D_v}$ are genus-minimizing Seifert surfaces~\cite[Thm.~9.3]{zbMATH03854015}, and in particular incompressible. So $T(D)$ is a tree of incompressible surfaces. We are going to use that in the upcoming proof of the following. %

\begin{theorem}
\label{thm:alternativevp}
Let $D$ be an alternative diagram for a link~$L$. If $L$ is not prime, then $D$ features an honest decomposition circle.
\end{theorem}

Let us make the notions appearing in the theorem precise.
A \emph{decomposition circle} of a link diagram $D$ is a circle intersecting $D$ transversely in two non-crossing points $x$,~$y$, such that on both sides of the circle, the subarc of $D$ connecting $x$,~$y$ has crossings. A \emph{decomposition sphere} of a link $L$ is a sphere intersecting $L$ transversely in two points~$x$,~$y$, such that on both sides of the sphere, there is no isotopy of $L$ moving the subarc of $L$ connecting $x$,~$y$ into the sphere.
The link $L$ is called \emph{prime} if it admits no decomposition sphere.
We call a decomposition circle \emph{honest} if it gives rise to a decomposition sphere.

Since alternating and positive links are in particular alternative,
we recover Menasco's and Ozawa's respective results:

\begin{corollary}[{Alternating links and positive links are visually prime~\cite{Menasco_84,zbMATH01807859}}]\label{cor:poslinksarevisprime}

Let $D$ be a positive or alternating diagram for a link~$L$. If $L$ is not prime, then $D$ features an honest decomposition circle.
\end{corollary}

Our proof of \Cref{thm:alternativevp} relies as base input on the visual primeness of special alternating links.
Since those are both alternating and positive (or negative),
this is a special case of both Menasco's and of Ozawa's theorem.
However, there is a purely algebraic proof of the result using symmetrized Seifert forms~\cite{zbMATH07931941}.
In particular, our reproof of \Cref{cor:poslinksarevisprime} avoids the geometric analysis of incompressible surfaces as done in \cite{Menasco_84} and in \cite{zbMATH01807859} and provides a unified approach. We state the base input as a lemma.

\begin{lemma}[Special alternating links are visually prime \cite{Menasco_84,zbMATH01807859,zbMATH07931941}]\label{lem:specialaltvp}
Let $D$ be a special alternating diagram for a link~$L$. If $L$ is not prime, then $D$ features an honest decomposition circle.
\end{lemma}

To summarize, our proof of \cref{thm:alternativevp} uses our understanding of fixed arcs in trees of incompressible veering surfaces, to reduce visual primeness of alternative links to the comparatively simple \cref{lem:specialaltvp},
which is a consequence of decomposing symmetrized Seifert forms.
As a last ingredient for the proof of \cref{thm:alternativevp}, we note that it suffices to restrict to link diagrams that are reduced and non-split.
\begin{lemma}\label{lem:nonsplit}
\begin{enumerate}[label=(\roman*)]
\item Let $D$ be a diagram of a non-prime link.
Then there is a split component of $D$ that is likewise a diagram of a non-prime link.
\item Let $D'$ be obtained from $D$ by removing a nugatory crossing by a twist. If $D'$ has an honest decomposition circle, then so does~$D$.
\end{enumerate}
\end{lemma}
\begin{proof}
A split union of links is prime if and only if all of the individual links are. This immediately implies~(i).

Let us prove (ii). By definition, there is a circle $C$ that intersects $D$ in the nugatory crossing, and nowhere else. Let $D'$ be obtained from $D$ by twisting the inside of $C$.
Let $C'$ be an honest decomposition circle of $D'$.
If $C' = C$, then a small isotopy making $C$ transverse to $D$ produces an honest decomposition circle of $D$.
Otherwise, $C'$ can be made disjoint from $C$. If $C'$ is on the outside of $C$, then $C'$ is an honest decomposition circle for $D$. If $C'$ is on the inside of $C$, then twisting it yields an honest decomposition circle for~$D$.
\end{proof}

    \begin{figure}[t]
    \centering
    \includegraphics[width=0.8\linewidth]{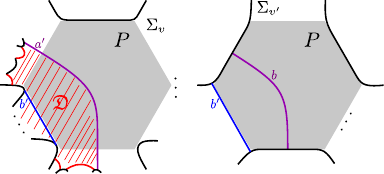}
    \caption{On the left, we have that the arc $b'$ is fixed in $\Sigma_v$ because it is contained in $\mathfrak{D}$. On the right, $b'$ is obtained from $b$ by an isotopy on the boundary, so it is also fixed in $\Sigma_{v'}$. Note that $b'$ could be boundary parallel in $\Sigma_v$, but not in $\Sigma_{v'}$ as there must be---at least---one side of $P$ between its endpoints.}
\label{fig:fixed_in_P}
    \end{figure}
\begin{proof} [Proof of \Cref{thm:alternativevp}]
By \cref{lem:nonsplit}, it suffices to prove the statement
for a non-split reduced alternative diagram~$D$.
 Let $T = T(D)$ be the corresponding tree of incompressible surfaces, see \cref{def:td}.
Let $L$ be the link given by $D$, that is the link bounding the Seifert surface $\Sigma\coloneqq \Sigma_D = \Sigma_{T} \subset M_T=S^3$.
Assume that $L$ is not prime; in other words, $\Sigma$ has an essential (isotopy class) of a fixed arc $a\in \mathcal{A}(\Sigma)$.
Our proof strategy is to construct from $a$ another essential fixed arc $c$ that is completely contained in a single Seifert disk.
After an isotopy, $c$ projects to an arc in $\mathbb{R}^2$ with interior disjoint from~$D$. Since $c$ separates, this arc in the plane may be completed to an honest decomposition circle as desired.

Note that all the surfaces $\Sigma_v$ at vertices $v$ of $T$ are strongly quasipositive or strongly quasinegative, so we can apply \Cref{prop:sqp=>rv} to conclude that every vertex of $T$ is either right- or left-veering. Then, by \Cref{thm:treeofveeringsurfaces}, all the intersections of $a$ with the surfaces $\Sigma_v$ are (isotopy classes of) fixed arcs.

There exists a vertex $v$ such that $a \cap \Sigma_v$ has an essential component~$a'$. Then, $a'$ is fixed in $\Sigma_v$.
Now, we have that $\Sigma_v$ is a surface coming from a special alternating diagram. By \Cref{lem:specialaltvp}, it decomposes as a boundary connected sum of irreducible non-disk surfaces %
along arcs $\{s_j\}_{j=1}^{n-1}$ contained in Seifert disks. Then, by \Cref{prop:connectsumcrit}, $a'$ cuts out a disk $\mathfrak{D}$ together with a subcollection of the~$s_j$. Since $a'$ is essential, the subcollection is not empty.

Consider the case that $a'$ does not intersect any summing region. Then any of the $s_j$ cutting out $\mathfrak{D}$ with $a'$ is in a Seifert disk that does not contain a summing region. It follows that $c = s_j$ is fixed, thus successfully executing our proof strategy.
We are left with the case 
that $a'$ intersects a summing region $P$ with a surface $\Sigma_{v'}$.
Let $b$ be a component of this intersection.
    By assumption, we can choose an arc $b'$ contained in $\mathfrak{D}$ with endpoints vertices of the sides of $P$ that $b$ intersects, see the left hand side of \Cref{fig:fixed_in_P}. Since $b' \subset \mathfrak{D}$, the discussion in \Cref{rmk:phitopartialopenbook} implies that $b'$ is fixed in $\Sigma_v$. Now, observe that, in $\Sigma_{v'}$, $b $ is a component of $a' \cap \Sigma_{v'}$. But this means it is also a component of $a \cap \Sigma_{v'}$, so by \Cref{thm:treeofveeringsurfaces} $b$ is fixed in $\Sigma_{v'}$. But $b'$ is obtained from $b$ by an free isotopy along the boundary, so $b'$ is also fixed in $\Sigma_{v'}$, see the right hand side of \Cref{fig:fixed_in_P}. Thus, $b'$ is an arc contained in $P$ that is fixed in both $\Sigma_v$ and $\Sigma_{v'}$. Thus $b'$ is fixed in~$\Sigma$.
Moreover, $P$ is contained in a Seifert disk, and thus so is~$c = b'$.
Once more, we have executed our proof strategy.
\end{proof}

\bibliographystyle{myamsalpha} 
\bibliography{main}
\end{document}